\documentclass[12pt]{article}
\usepackage{amssymb}
\usepackage{amsmath}
\usepackage{amsthm}
\usepackage{color}
\usepackage{comment}
\usepackage{geometry}		
\usepackage{graphicx}

\let\originalleft\left
\let\originalright\right
\renewcommand{\left}{\mathopen{}\mathclose\bgroup\originalleft}
\renewcommand{\right}{\aftergroup\egroup\originalright}

\begin{document}

\def\ee{\varepsilon}
\def\cM{\mathcal{M}}
\def\cO{\mathcal{O}}

\newcommand{\removableFootnote}[1]{}

\newtheorem{theorem}{Theorem}[section]
\newtheorem{conjecture}[theorem]{Conjecture}
\newtheorem{lemma}[theorem]{Lemma}
\newtheorem{proposition}[theorem]{Proposition}

\theoremstyle{definition}
\newtheorem{definition}{Definition}[section]

\title{
Dimension reduction for slow-fast, piecewise-smooth, continuous systems of ODEs.
}
\author{
D.J.W.~Simpson\\\\
Institute of Fundamental Sciences\\
Massey University\\
Palmerston North\\
New Zealand
}
\maketitle



\begin{abstract}
The limiting slow dynamics of slow-fast, piecewise-linear, continuous systems of ODEs occurs on critical manifolds that are piecewise-linear. At points of non-differentiability, such manifolds are not normally hyperbolic and so the fundamental results of geometric singular perturbation theory do not apply. In this paper it is shown that if the critical manifold is globally stable then the system is forward invariant in a neighbourhood of the critical manifold. It follows that in this neighbourhood the dynamics is given by a regular perturbation of the dynamics on the critical manifold and so dimension reduction can be achieved. If the attraction is instead non-global, additional dynamics involving canards may be generated. For boundary equilibrium bifurcations of piecewise-smooth, continuous systems, the results are used to establish a general methodology by which such bifurcations can be analysed. This approach is illustrated with a three-dimensional model of ocean circulation.
\end{abstract}

\section{Introduction}
\label{sec:intro}
\setcounter{equation}{0}

Bifurcation theory provides an elegant method by which local bifurcations of $n$-dimensional systems of ODEs
can be analysed by reducing the system to a small set of equations.
Essentially the pertinent dynamics occurs on a centre manifold
whose dimension depends only on the type of bifurcation.
The restriction of the system to the centre manifold produces a low-dimensional system
that describes the full system dynamics quantitatively \cite{Ku04}.

For piecewise-smooth ODE systems, however, bifurcations involving a switching manifold 
usually cannot be analysed in the same way because a centre manifold simply does not exist.
A vast mathematical theory exists for local bifurcations of piecewise-smooth systems \cite{DiBu08,MaLa12},
but much of it only applies to systems of one or two dimensions.

In view of the effectiveness of dimension reduction methods for smooth systems,
various limited forms of dimension reduction have been developed for piecewise-smooth systems.
Piecewise-linear (PWL) systems with an equilibrium on a switching manifold often have invariant cones.
If an invariant cone is attracting, the restriction of the system to the cone generates
a lower-dimensional system that describes the long-term dynamics \cite{Ku08,KuHo13}.
Codimension-two bifurcations of piecewise-smooth systems that involve a non-hyperbolic equilibrium,
can be partly understood by studying the dynamics on the
centre manifold of one smooth component of the system \cite{Si10}.
This approach can also be applied to maps \cite{CoDe10}.
Certain grazing bifurcations generate Poincar\'e maps with a square-root singularity
which stretches phase space in a particular direction
and so allows for a reduction to one dimension in some cases \cite{FrNo97}\removableFootnote{
Also section 4.3.2 of [DiBu08].
}.
In \cite{PaWi07} it is shown how various features of a five-dimensional impact oscillator model
allow for reduction to one dimension.
Also, at shrinking points of mode-locking regions of $n$-dimensional PWL maps,
one-dimensional centre manifolds enable reduction to a skew-sawtooth circle map \cite{Si16e}.

This paper concerns boundary equilibrium bifurcations (BEBs), also called discontinuous bifurcations,
of piecewise-smooth continuous systems at which an equilibrium collides with a switching manifold.
In a neighbourhood of a BEB, the system is PWL to leading order\removableFootnote{
PWL continuous systems are also used as mathematical models \cite{LeVa98},
particularly for circuits \cite{Ch94} and neurons \cite{Co08}.
}.
The {\em local} dynamics of a BEB are governed by the {\em global} dynamics of the corresponding PWL system.
Such dynamics can be chaotic \cite{Sp81,Si16c}, or otherwise inherently high-dimensional,
in which case dimension reduction may not be possible.

However, such global dynamics may be captured
by a lower-dimensional set of equations if the system is slow-fast.
For smooth slow-fast systems, dimension reduction is achieved
via techniques in geometric singular perturbation theory \cite{Fe79,Jo95b,Ku15}\removableFootnote{
See also \cite{Wi94}.
Originally done in \cite{Fe71}.
}.
For an $n$-dimensional system with $k$ fast variables and a time-scale separation parameter $\ee$,
the $\ee \to 0$ limit defines an $(n-k)$-dimensional critical manifold $\cM_0$.
Fenichel's theorem \cite{Fe71} tells us that if $\cM_0$ is normally hyperbolic,
then for small $\ee > 0$ there exists an $(n-k)$-dimensional, locally invariant, slow manifold $\cM_\ee$
that is an $\cO(\ee)$ perturbation of $\cM_0$ and diffeomorphic to $\cM_0$.
The dynamics on $\cM_\ee$ is consequently an $\cO(\ee)$ perturbation of the dynamics on $\cM_0$.

Fenichel's theorem can be applied to slow-fast, PWL, continuous systems,
as long as we only consider subsets of phase space that do not contain a switching manifold.
In this way we can obtain a linear, locally invariant, slow manifold $\cM_\ee$
for each linear component of the system\removableFootnote{
I have to be careful with the wording here because $\cM_\ee$ is non-unique.
}.
Each $\cM_\ee$ is aligned with the slow eigenspaces of the Jacobian matrix of the relevant component.
The system may have both stable and unstable slow manifolds $\cM_\ee$,
and canards that evolve on both manifolds \cite{PoRa11,PrTe13,PrTe16}.
Each $\cM_\ee$ can be extended to form a global invariant manifold.
This is achieved in \cite{KoGl11} in the lowest-dimensional case, $(n,k) = (2,1)$,
to explain the origin of oscillatory motion.
In higher dimensions, the slow manifolds may have a complicated global structure.
Far from the origin, two-piece PWL systems are approximately homogeneous,
and homogeneous PWL systems, which need not be slow-fast, can exhibit invariant cones \cite{CaFr05,CaFe12}.
	
The main result of this paper concerns $n$-dimensional, PWL, continuous systems
with $k$ fast variables and a time-scale separation parameter $\ee$.
On the switching manifold, the critical manifold $\cM_0$ is continuous but not normally hyperbolic.
We show that if $\cM_0$ is attracting, in a certain sense,
then, as with Fenichel's theorem, the dynamics for $\ee > 0$ is an $\cO(\ee)$ perturbation of the dynamics on $\cM_0$.
However, since it is not clear that a slow manifold $\cM_\ee$ diffeomorphic to $\cM_0$ always exists,
this is achieved by constructing an $\cO(\ee)$ neighbourhood of $\cM_0$ that is forward invariant.

We propose that this result be used to analyse BEBs
in systems for which it may not be apparent which, or even how many, variables are fast.
Generic BEBs are characterised by two sets of $n$ eigenvalues.
If, in each set, $k$ eigenvalues are much larger than the rest,
the above result tells us that the $(n-k)$-dimensional system formed from the remaining eigenvalues
should provide a good approximation to the dynamics of the full system.
	
The remainder of this paper is organised as follows.
In \S\ref{sec:trans} we introduce the ``slow-fast observer canonical form'' (SFOCF).
Whereas the observer canonical form provides a normal form for BEBs,
this new form is better suited for studying BEBs in slow-fast systems.
We derive explicit formulas for the coordinate change
from an arbitrary PWL system to the SFOCF.
In \S\ref{sec:dimred} we consider the SFOCF in the $\ee \to 0$ limit.
We describe the reduced system, the layer equations, and the critical manifold $\cM_0$,
and show how the reduced system is connected to the slow eigenvalues of the BEB.

In \S\ref{sec:mainthm} we derive sufficient conditions
for the existence of a forward invariant region near $\cM_0$ (Theorem \ref{th:mainBound}).
We also show that if global stability is not satisfied,
then orbits can diverge or exhibit complicated behaviour.
In \S\ref{sec:app} we introduce a general method by which BEBs can be analysed through dimension reduction
and apply this method to a three-dimensional model of ocean circulation.
Finally \S\ref{sec:conc} provides conclusions and an outlook for future studies.

\section{Coordinate transformations}
\label{sec:trans}
\setcounter{equation}{0}

We consider ODE systems of the form
\begin{equation}
\begin{split}
\dot{x} &= \begin{cases}
f_L(x,y;\ee), & h(x,y) \le 0, \\
f_R(x,y;\ee), & h(x,y) \ge 0,
\end{cases} \\
\dot{y} &= \begin{cases}
\ee g_L(x,y;\ee), & h(x,y) \le 0, \\
\ee g_R(x,y;\ee), & h(x,y) \ge 0,
\end{cases}
\end{split}
\label{eq:pws}
\end{equation}
where $x \in \mathbb{R}^k$ is the fast variable,
$y \in \mathbb{R}^{n-k}$ is the slow variable,
and $\ee \ge 0$ is the time-scale separation parameter.
This is a two-piece, piecewise-smooth system with switching manifold $h(x,y) = 0$.
We assume that $f_L$, $f_R$, $g_L$, $g_R$, and $h$ are twice differentiable,
and that the right-hand-side of \eqref{eq:pws} is continuous on the switching manifold.

Next we change coordinates to simplify the switching condition
and approximate \eqref{eq:pws} with a PWL system in a neighbourhood of a BEB, \S\ref{sub:beb}.
We then review the observer canonical form in \S\ref{sub:companion} and \S\ref{sub:ocf},
and introduce the SFOCF in \S\ref{sub:sfocf}.

\subsection{A general piecewise-linear form}
\label{sub:beb}


Suppose $h(0,0) = 0$ and that the $n$-dimensional gradient vector $\nabla h$ is not the zero vector.
That is, locally, the switching manifold of \eqref{eq:pws}
is a smooth $(n-1)$-dimensional manifold intersecting the origin.
We further suppose that at the origin the switching manifold is not tangent to all fast directions.
That is, $\nabla h$ is non-zero in at least one of its first $k$ components.
By reordering the components of $x$, we can assume that the first component of $\nabla h$ is non-zero,
that is $\frac{\partial h}{\partial x_1} \ne 0$.

Let $\tilde{x} = [h(x,y),x_2,\ldots,x_k]^{\sf T}$ and $\tilde{y} = y$.
The transformation $(x,y) \to (\tilde{x},\tilde{y})$ is invertible
because $\frac{\partial h}{\partial x_1} \ne 0$.
Since $\tilde{x}_1$ is a fast variable,
the transformed system has the same slow-fast form as \eqref{eq:pws}
except the switching manifold is simply $\tilde{x}_1 = 0$\removableFootnote{
This should all be okay if $h$ is $C^2$.
However, if $f_L,\ldots,g_R$ are $C^K$ ($K \ge 2$) in the original system,
they are only $C^2$ in the transformed system, I think.
}.

Now suppose that the system has a BEB
at the origin when a parameter $\tilde{\mu}$ is zero.
Structurally stable dynamics of the system near the bifurcation are captured by
its PWL approximation.
This is obtained by replacing each smooth component of the system
with the linear terms of its Taylor expansion centred at $(\tilde{x},\tilde{y};\tilde{\mu}) = (0,0;0)$.
This has the form
\begin{equation}
\begin{split}
\dot{\tilde{x}} &= \begin{cases}
U_L(\ee) \left[ \begin{array}{c} \tilde{x} \\ \hline \tilde{y} \end{array} \right] + q(\ee) \tilde{\mu}, & \tilde{x}_1 \le 0, \\
U_R(\ee) \left[ \begin{array}{c} \tilde{x} \\ \hline \tilde{y} \end{array} \right] + q(\ee) \tilde{\mu}, & \tilde{x}_1 \ge 0,
\end{cases} \\
\dot{\tilde{y}} &= \begin{cases}
\ee V_L(\ee) \left[ \begin{array}{c} \tilde{x} \\ \hline \tilde{y} \end{array} \right] + \ee r(\ee) \tilde{\mu}, & \tilde{x}_1 \le 0, \\
\ee V_R(\ee) \left[ \begin{array}{c} \tilde{x} \\ \hline \tilde{y} \end{array} \right] + \ee r(\ee) \tilde{\mu}, & \tilde{x}_1 \ge 0,
\end{cases}
\end{split}
\label{eq:sfpwl}
\end{equation}
where $U_L$ and $U_R$ are $k \times n$ matrices,
$V_L$ and $V_R$ are $(n-k) \times n$ matrices,
$q \in \mathbb{R}^k$,
and $r \in \mathbb{R}^{n-k}$.
To simplify the notation we write
\begin{align}
\tilde{z} &= \left[ \begin{array}{c} \tilde{x} \\ \hline \tilde{y} \end{array} \right], &
P_L(\ee) &= \left[ \begin{array}{c} U_L(\ee) \\ \hline \ee V_L(\ee) \end{array} \right], &
P_R(\ee) &= \left[ \begin{array}{c} U_R(\ee) \\ \hline \ee V_R(\ee) \end{array} \right], &
c(\ee) &= \left[ \begin{array}{c} q(\ee) \\ \hline \ee r(\ee) \end{array} \right], &
\nonumber
\end{align}
with which \eqref{eq:sfpwl} becomes
\begin{equation}
\dot{\tilde{z}} = \begin{cases}
P_L(\ee) \tilde{z} + c(\ee) \tilde{\mu}, & \tilde{x}_1 \le 0, \\
P_R(\ee) \tilde{z} + c(\ee) \tilde{\mu}, & \tilde{x}_1 \ge 0.
\end{cases}
\label{eq:sfpwl2}
\end{equation}
By continuity, at $\tilde{x} = 0$ the matrices $P_L(\ee)$ and $P_R(\ee)$ only differ in their first columns.

\subsection{Companion matrices}
\label{sub:companion}

Here we clarify notation regarding basis vectors and companion matrices.

Given $m \ge 1$,
let $e_1,\ldots,e_m$ denote the standard basis vectors of $\mathbb{R}^m$.
Below we work in different dimensions but the dimensions of the basis vectors
should be clear from the context.
The identity matrix is $I_m = \big[ e_1 \cdots e_m \big]$,
and we write $J_m = \big[ 0~e_1 \cdots e_{m-1} \big]$ (with $J_1 = 0$).

A matrix of the form $J_m - p e_1^{\sf T}$, where $p \in \mathbb{R}^m$, is called a {\em companion matrix}.
Companion matrices are convenient in that the components of $p$
provide the coefficients of the characteristic polynomial:
\begin{equation}
\det \left( \lambda I_m - \left( J_m - p e_1^{\sf T} \right) \right)
= \lambda^m + p_1 \lambda^{m-1} + \cdots + p_{m-1} \lambda + p_m \,.
\label{eq:charPoly}
\end{equation}

\subsection{The observer canonical form}
\label{sub:ocf}

The observer canonical form is a PWL system involving companion matrices.
In $n$ dimensions it may be written as
\begin{equation}
\dot{z} = \begin{cases}
\left( J_n - p^L e_1^{\sf T} \right) z + e_n \mu, & z_1 \le 0, \\
\left( J_n - p^R e_1^{\sf T} \right) z + e_n \mu, & z_1 \ge 0,
\end{cases}
\label{eq:ocf}
\end{equation}
where $p^L, p^R \in \mathbb{R}^n$.
The following result gives conditions under which
the general PWL system \eqref{eq:sfpwl2} can be transformed to \eqref{eq:ocf},
and provides explicit formulas for the transformation.
Here the slow-fast form of \eqref{eq:sfpwl2} is not important
and the $\ee$-dependency can be ignored.

\begin{proposition}
Consider a system of the form \eqref{eq:sfpwl2}.
Let $p^L$ be the vector whose components are the coefficients of the
characteristic polynomial of $P_L$.
Let
\begin{align}
\Psi &= \left[ \begin{array}{cccc}
1 \\
p^L_1 & 1 \\
\vdots & \ddots & \ddots \\
p^L_{n-1} & \cdots & p^L_1 & 1
\end{array} \right], &
\Phi &= \left[ \begin{array}{c}
e_1^{\sf T} \\
e_1^{\sf T} P_L \\
\vdots \\
e_1^{\sf T} P_L^{n-1}
\end{array} \right],
\label{eq:PsiPhi}
\end{align}
and
\begin{align}
Q &= \Psi \Phi, &
d &= J_n^{\sf T} Q c, &
s &= e_n^{\sf T} Q c.
\label{eq:Qds}
\end{align}
If $\Phi$ is non-singular and $s \ne 0$, then the change of variables
\begin{align}
z &= Q \tilde{z} + d \tilde{\mu}, &
\mu &= s \tilde{\mu},
\label{eq:transocf}
\end{align}
transforms \eqref{eq:sfpwl2} into \eqref{eq:ocf}\removableFootnote{
I have implemented this in {\sc transToOCF2.m}.
}.
\label{pr:ocf}
\end{proposition}

Proposition \ref{pr:ocf} is based on standard techniques in control theory \cite{DiBu08}.
We provide a full proof of Proposition \ref{pr:ocf} in Appendix \ref{app:ocf}
because some formulas developed in the proof
are used below to prove the result of the next section\removableFootnote{
I have been unable to find a nice proof of Proposition \ref{pr:ocf},
although, as shown in Appendix \ref{app:ocf}, it is not actually that difficult.
Perhaps more difficult is showing that the given condition is necessary.
Proofs of using alternate form for $Q$ are given in \cite{CaFr02,Si10}.
}\removableFootnote{
Notice that $e_1^{\sf T} Q = e_1^{\sf T}$ and $e_1^{\sf T} d = 0$.
Therefore $z_1 = \tilde{z}_1$.
If $\Phi$ is singular, then \eqref{eq:sfpwl2} cannot be transformed to \eqref{eq:ocf}.
If $s \ne 0$, then $\tilde{\mu}$ does not unfold the BEB in a generic fashion
in that \eqref{eq:sfpwl2} has an equilibrium on the switching manifold for all $\tilde{\mu} \in \mathbb{R}$.
}.

\subsection{The slow-fast observer canonical form}
\label{sub:sfocf}

We define the slow-fast observer canonical form (SFOCF) as 
\begin{equation}
\dot{z} = \begin{cases}
C_L(\ee) z + \ee e_n \mu, & z_1 \le 0, \\
C_R(\ee) z + \ee e_n \mu, & z_1 \ge 0,
\end{cases}
\label{eq:sfocf}
\end{equation}
where
\begin{equation}
C_X(\ee) = \left[ \begin{array}{cccc|cccc}
-a^X_1(\ee) & 1 && \\
\vdots && \ddots & \\
&&& 1 \\
-a^X_k(\ee) &&&& 1 \\ \hline
-\ee b^X_1(\ee) &&&&& \ee \\
\vdots &&&&&& \ddots \\
&&&&&&& \ee \\
-\ee b^X_{n-k}(\ee) &&&
\end{array} \right],
\label{eq:CX}
\end{equation}
and $a^X \in \mathbb{R}^k$, $b^X \in \mathbb{R}^{n-k}$, for $X = L,R$.
It has the both slow-fast form \eqref{eq:sfpwl2}
and zeros in the same entries as the observer canonical form \eqref{eq:ocf}.
In order to transform \eqref{eq:sfpwl2} into \eqref{eq:sfocf}, we
combine the change of variables of Proposition \ref{pr:ocf} with the observation that 
$E C_X E^{-1}$ is a companion matrix, where
\begin{equation}
E(\ee) = \left[ \begin{array}{ccc|cccc}
1 && \\
& \ddots & \\
&& 1 \\ \hline
&&& 1 \\
&&&& \ee \\
&&&&& \ddots \\
&&&&&& \ee^{n-k-1}
\end{array} \right].
\label{eq:E}
\end{equation}

\begin{proposition}
Let $Q$, $d$, and $s$ be defined as in \eqref{eq:Qds}.
If $\Phi$ is non-singular and $s \ne 0$, then the change of variables
\begin{align}
z &= E^{-1} \left( Q \tilde{z} + d \tilde{\mu} \right), &
\mu &= \frac{s}{\ee^{n-k}} \tilde{\mu},
\label{eq:transsfocf}
\end{align}
transforms \eqref{eq:sfpwl2} into \eqref{eq:sfocf}\removableFootnote{
I have implemented this in {\sc transToSFOCF2.m}.
}.
\label{pr:sfocf}
\end{proposition}

Proposition \ref{pr:sfocf} is proved in Appendix \ref{app:sfocf} by direct calculations\removableFootnote{
If the system has the particular slow-fast form described earlier,
then the transformation from $\tilde{z}$ to $z$ is order $1$.
This is intuitive but tricky to prove.
I've decided not to include a proof because it not worth the effort.
We have to directly evaluate $\hat{Q} = E^{-1} \Psi \Phi$
because there are many similarity transforms from $P_L$ to $C_L$.
Also we can't simply set $\ee = 0$ because then we may have $\det(\Phi) = 0$.

My proof in {\sc BEBDimReduction\_old.tex} using old coordinates and notation seems like the best approach.
Here's how I expect the proof to go:

One first applies a near-similarity transform to put $P_L$ in block diagonal form.
Then powers of $P_L$ also have a block diagonal form.
By evaluating $E^{-1} \Psi \Phi$ we find that, in block form, the top half of $\hat{Q}$ is $\cO(1)$,
and, by applying the Cayley-Hamilton to the top-left block of $P_L$, the bottom-left block of $\hat{Q}$ is $\cO(\ee)$.
This gives a stronger result than intended, but is needed to deal with the bottom-right block of $\hat{Q}$:
via the formula $C_L \hat{Q} = \hat{Q} P_L$, with each matrix in block form,
since the $\ee$'s in $C_L$ on the left hand side cancel with the $\ee$'s in $\hat{Q}$ on the right hand side,
we deduce that the bottom-right block of $\hat{Q}$ is $\cO(1)$.
}.

\section{Dynamics in the slow-fast limit}
\label{sec:dimred}
\setcounter{equation}{0}

The SFOCF \eqref{eq:sfocf} separates into its fast and slow components as
\begin{equation}
\begin{split}
\dot{x} &= \begin{cases}
\left( J_k - a^L(\ee) e_1^{\sf T} \right) x + e_k y_1 \,, & x_1 \le 0, \\
\left( J_k - a^R(\ee) e_1^{\sf T} \right) x + e_k y_1 \,, & x_1 \ge 0,
\end{cases} \\
\dot{y} &= \begin{cases}
\ee \left( -b^L(\ee) x_1 + J_{n-k} y + e_{n-k} \mu \right), & x_1 \le 0, \\ 
\ee \left( -b^R(\ee) x_1 + J_{n-k} y + e_{n-k} \mu \right), & x_1 \ge 0.
\end{cases}
\end{split}
\label{eq:sfocf2}
\end{equation}
By taking $\ee \to 0$, we obtain the {\em layer equations}
\begin{align}
\dot{x} &= \begin{cases}
A_L x + e_k y_1 \,, & x_1 \le 0, \\
A_R x + e_k y_1 \,, & x_1 \ge 0,
\end{cases}
\label{eq:limitingFastx} \\
\dot{y} &= 0,
\label{eq:limitingFasty}
\end{align}
where
\begin{align}
A_L &= J_k - a^L(0) e_1^{\sf T}, &
A_R &= J_k - a^R(0) e_1^{\sf T}.
\label{eq:ALAR}
\end{align}
Alternatively, on the slow time-scale $\tau = \ee t$, the limit $\ee \to 0$ produces the {\em reduced system}
\begin{align}
0 &= \begin{cases}
A_L x + e_k y_1 \,, & x_1 \le 0, \\
A_R x + e_k y_1 \,, & x_1 \ge 0,
\end{cases}
\label{eq:limitingSlowx0} \\
\frac{d y}{d \tau} &= \begin{cases}
-b^L(0) x_1 + J_{n-k} y + e_{n-k} \mu, & x_1 \le 0, \\
-b^R(0) x_1 + J_{n-k} y + e_{n-k} \mu, & x_1 \ge 0.
\end{cases}
\label{eq:limitingSlowy0}
\end{align}
In this section we first derive the critical manifold $\cM_0$, \S\ref{sub:M0}.
We then discuss the stability of $\cM_0$, \S\ref{sub:expstab},
and describe the reduced system restricted to $\cM_0$, \S\ref{sub:reducedSystem}.

\subsection{The critical manifold}
\label{sub:M0}

\begin{figure}[b!]
\begin{center}
\setlength{\unitlength}{1cm}
\begin{picture}(12.5,4.8)
\put(0,0){\includegraphics[height=4.8cm]{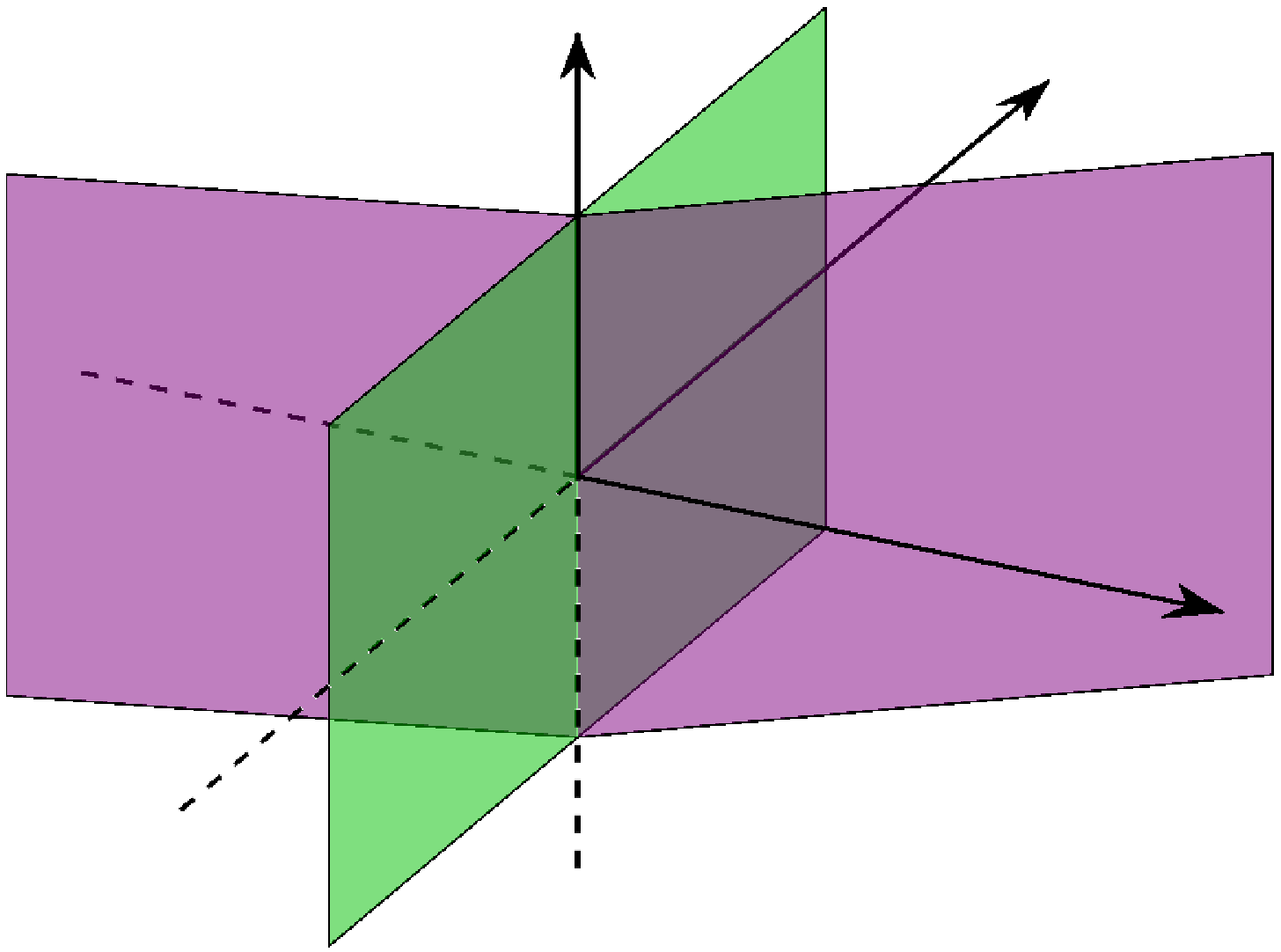}}
\put(6.9,0){\includegraphics[height=4.8cm]{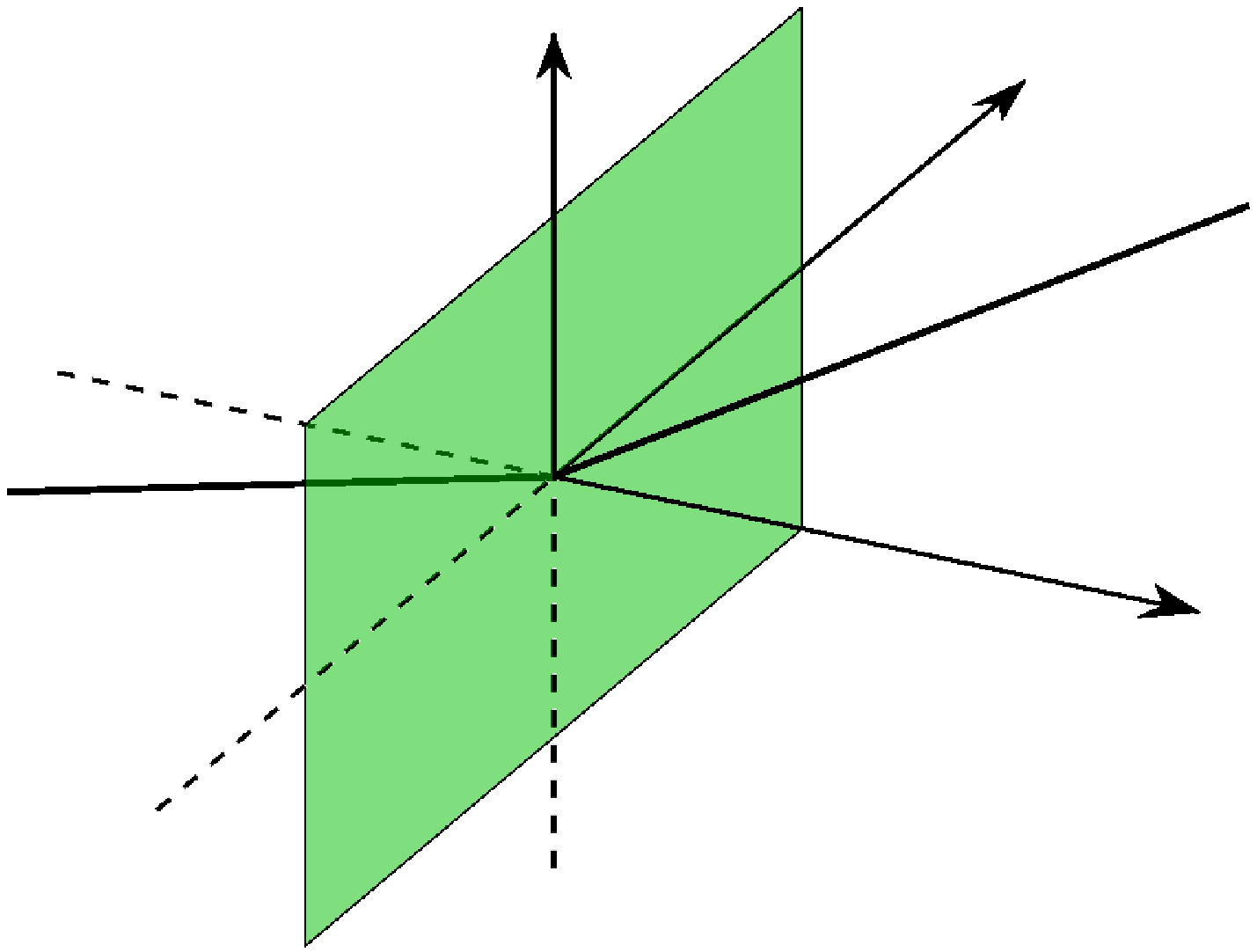}}
\put(.2,4.7){\sf \bfseries A}
\put(7.1,4.7){\sf \bfseries B}
\put(5.97,1.84){\scriptsize $x_1$}
\put(5.22,4.13){\scriptsize $y_1$}
\put(3,4.4){\scriptsize $y_2$}
\put(.7,3.93){\scriptsize $\cM_0$}
\put(12.87,1.84){\scriptsize $x_1$}
\put(12.12,4.13){\scriptsize $x_2$}
\put(9.9,4.4){\scriptsize $y_1$}
\put(12.6,3.23){\scriptsize $\cM_0$}
\end{picture}
\caption{
Sketches of the $(n-k)$-dimensional critical manifold $\cM_0$.
In panel A, $(n,k) = (3,1)$ (one fast variable, two slow variables).
In panel B, $(n,k) = (3,2)$ (two fast variables, one slow variable).
In each panel the green surface is the switching manifold $x_1 = 0$.
\label{fig:schemCriticalManifold}
} 
\end{center}
\end{figure}

The critical manifold $\cM_0$ is defined by the algebraic constraint \eqref{eq:limitingSlowx0},
see Fig.~\ref{fig:schemCriticalManifold}.
For the fast system \eqref{eq:limitingFastx}, it is a surface of equilibria.
Here we compute these equilibria and express $\cM_0$ as a function of $y$.

Suppose $A_L$ and $A_R$ are non-singular.
Then the components of \eqref{eq:limitingFastx}
have unique equilibria, $x^L(y_1) = -A_L^{-1} e_k y_1$ and $x^R(y_1) = -A_R^{-1} e_k y_1$,
for any $y_1 \in \mathbb{R}$.
Since $A_L$ and $A_R$ are companion matrices \eqref{eq:ALAR},
their inverses have a simple form, and
\begin{align}
x^L(y_1) &= \frac{1}{a^L_k(0)} \left[ \begin{array}{c}
1 \\ a^L_1(0) \\ \vdots \\ a^L_{k-1}(0) \end{array} \right] y_1 \,, &
x^R(y_1) &= \frac{1}{a^R_k(0)} \left[ \begin{array}{c}
1 \\ a^R_1(0) \\ \vdots \\ a^R_{k-1}(0) \end{array} \right] y_1 \,,
\label{eq:xLxR}
\end{align}
where $a^L_k(0) = (-1)^k \det(A_L)$ and $a^R_k(0) = (-1)^k \det(A_R)$ are non-zero by assumption.

Below we assume all eigenvalues of $A_L$ and $A_R$ have negative real-part,
as this is needed to ensure that the critical manifold is stable.
Here we show that this condition also ensures that
\eqref{eq:limitingFastx} has a unique equilibrium for all $y_1 \in \mathbb{R}$.

\begin{lemma}
Suppose all eigenvalues of $A_L$ and $A_R$ have negative real part.
Then $a^L_k(0), a^R_k(0) > 0$, and \eqref{eq:limitingFastx} has the unique equilibrium
\begin{equation}
H(y_1) = \begin{cases}
x^L(y_1), & y_1 \le 0, \\
x^R(y_1), & y_1 \ge 0,
\end{cases}
\label{eq:H}
\end{equation}
for all $y_1 \in \mathbb{R}$.
\label{le:eq}
\end{lemma}

\begin{proof}
Since $A_L$ has the companion matrix form \eqref{eq:ALAR}, its characteristic polynomial is
\begin{equation}
\det \left( \lambda I - A_L \right)
= \lambda^k + a^L_1(0) \lambda^{k-1} + \cdots + a^L_{k-1}(0) \lambda + a^L_k(0).
\nonumber
\end{equation}
By assumption there are no positive values of $\lambda$ for which $\det \left( \lambda I - A_L \right) = 0$.
Since $\det \left( \lambda I - A_L \right)$ is positive for large values of $\lambda$,
it is also positive with $\lambda = 0$.
That is, $a^L_k(0) > 0$.
Also $a^R_k(0) > 0$ for the same reasons.

Since $A_L$ and $A_R$ are non-singular,
$x^L(y_1)$ and $x^R(y_1)$ are the only potential equilibria of \eqref{eq:limitingFastx}.
The point $x^L(y_1)$ is an equilibrium of \eqref{eq:limitingFastx}
if its first component, $\frac{y_1}{a^L_k(0)}$, is less than or equal to zero.
Since $a^L_k(0) > 0$, this is the case if and only if $y_1 \le 0$.
Similarly $x^R(y_1)$ is an equilibrium of \eqref{eq:limitingFastx} if and only if $y_1 \ge 0$.
Also $x^L(y_1) = x^R(y_1)$ if $y_1 = 0$.
Therefore, $H(y_1)$ is the unique equilibrium of \eqref{eq:limitingFastx} for any $y_1 \in \mathbb{R}$.
\end{proof}

From Lemma \ref{le:eq} we can write the critical manifold as
\begin{equation}
\cM_0 = \left\{ \left[ \begin{array}{c} x \\ \hline y \end{array} \right] \,\middle|\,
x = H(y_1),\, y \in \mathbb{R}^{n-k} \right\}.
\label{eq:M0}
\end{equation}
In view of the above eigenvalue assumption,
every point on $\cM_0$ that does not belong to the switching manifold, $x_1 = 0$,
is a hyperbolic equilibrium of \eqref{eq:limitingFastx}.
Therefore, away from the switching manifold $\cM_0$ is normally hyperbolic.

\subsection{Stability of the critical manifold}
\label{sub:expstab}

Here we discuss the stability of the equilibrium $H(y_1)$ of \eqref{eq:limitingFastx}
subject to the assumption that all eigenvalues of $A_L$ and $A_R$ have negative real part.

If $y_1 \ne 0$, then $H(y_1)$ does not belong to the switching manifold of \eqref{eq:limitingFastx}
and is asymptotically stable due to the eigenvalue assumption.
If $y_1 = 0$, then $H(y_1)$ belongs to the switching manifold and
its stability is not easily characterised in terms of the eigenvalues of $A_L$ and $A_R$.
Indeed $H(0)$ can be unstable \cite{CaFr06},
although this requires \eqref{eq:limitingFastx} to be at least three-dimensional ($k \ge 3$).

In order to show that the dynamics of the full system \eqref{eq:sfocf2} stays near $\cM_0$ when $\ee > 0$,
we require $H(y_1)$ to satisfy a particularly strong form of stability.
Recall, an equilibrium is said to be {\em globally asymptotically stable}
if it is asymptotically stable and its basin of attraction is the whole space.
Here we denote the flow of \eqref{eq:limitingFastx} by $\phi_t(x;y_1)$.

\begin{definition}
Suppose all eigenvalues of $A_L$ and $A_R$ have negative real part.
We say that $\cM_0$ is {\em globally stable} if
$H(y_1)$ is a globally asymptotically stable equilibrium of \eqref{eq:limitingFastx} for all $y_1 \in \mathbb{R}$.
We say that $\cM_0$ is {\em globally exponentially stable} if
there exist $\alpha \ge 1$ and $\beta > 0$ such that
\begin{equation}
\left\| \phi_t(x;y_1) - H(y_1) \right\| \le
\alpha {\rm e}^{-\beta t} \left\| x - H(y_1) \right\|, \quad
{\rm for~all~} y_1 \in \mathbb{R},
{\rm ~all~} x \in \mathbb{R}^k,
{\rm ~and~all~} t \ge 0.
\label{eq:expStab}
\end{equation}
\label{df:expStab}
\end{definition}

The next result (proved in Appendix \ref{app:expStab}) shows that these two forms of stability are equivalent.
That global exponential stability implies global stability is trivial;
the converse is an artifact of the PWL nature of \eqref{eq:limitingFastx}.

\begin{lemma}
The critical manifold $\cM_0$ is globally stable if and only if it is globally exponentially stable.
\label{le:expStab}
\end{lemma}

If all eigenvalues of $A_L$ and $A_R$ have negative real part,
then $H(y_1)$ is asymptotically stable for all $y_1 \ne 0$.
If $H(y_1)$ is also asymptotically stable for $y_1 = 0$,
then, with $y_1 \ne 0$, orbits far from $H(y_1)$ travel inwards because the large-scale dynamics
are approximated by the system with $y_1 = 0$.
In the simplest scenario, $H(y_1)$ is globally asymptotically stable.
Indeed numerical explorations have failed to find other invariant sets,
and so here we conjecture that this must be the case.
If true, this result provides us with a weaker set of conditions that ensures $\cM_0$ is globally stable,
but we have been unable to prove it because the intermediate-scale dynamics is highly nonlinear.

\begin{conjecture}
Suppose all eigenvalues of $A_L$ and $A_R$ have negative real part.
If $H(y_1)$ is an asymptotically stable equilibrium of \eqref{eq:limitingFastx} for $y_1 = 0$,
then $\cM_0$ is globally stable.
\label{cj:globalStability}
\end{conjecture}

\subsection{The reduced system restricted to the critical manifold}
\label{sub:reducedSystem}

By \eqref{eq:xLxR}, on $\cM_0$ the first component of $x$ is given by
\begin{equation}
x_1 = \begin{cases}
\frac{1}{a^L_k(0)} y_1 \,, & y_1 \le 0, \\
\frac{1}{a^R_k(0)} y_1 \,, & y_1 \ge 0.
\end{cases}
\nonumber
\end{equation}
By substituting this into \eqref{eq:limitingSlowy0} we can rewrite the reduced system as
\begin{align}
x(\tau) &= H(y_1(\tau)),
\label{eq:limitingSlowx} \\
\frac{d y}{d \tau} &= \begin{cases} 
B_L y + e_{n-k} \mu, & y_1 \le 0, \\
B_R y + e_{n-k} \mu, & y_1 \ge 0,
\end{cases}
\label{eq:limitingSlowy}
\end{align}
where
\begin{align}
B_L &= J_{n-k} - \frac{b^L(0)}{a^L_k(0)} e_1^{\sf T}, &
B_R &= J_{n-k} - \frac{b^R(0)}{a^R_k(0)} e_1^{\sf T}.
\label{eq:BLBR}
\end{align}

Here we show that, for both $X = L$ and $X = R$,
the eigenvalues of $C_X(\ee)$, in the original form \eqref{eq:sfocf},
are those of $A_X$ and $\ee B_X$, to leading order.

\begin{lemma}
Suppose all eigenvalues of $A_L$ and $A_R$ have negative real part.
Then, for all $\lambda \in \mathbb{C}$,
\begin{equation}
\det \left( \lambda I - C_X(\ee) \right)
= \det \left( \lambda I - A_X + \cO(\ee) \right)
\det \left( \lambda I - \ee B_X + \cO \left( \ee^2 \right) \right),
\label{eq:eigsABC}
\end{equation}
for $X = L,R$.
\label{le:eigsABC}
\end{lemma}

\begin{proof}
By evaluating the companion matrix $E C_X E^{-1}$,
where $E$ is given by \eqref{eq:E},
we find that the characteristic polynomial of $C_X$ is
\begin{align}
\det(\lambda I - C_X(\ee)) &= \lambda^n + a^X_1(\ee) \lambda^{n-1} + \cdots
+ a^X_k(\ee) \lambda^{n-k} \nonumber \\
&\quad+ \ee b^X_1(\ee) \lambda^{n-k-1} + \cdots
+ \ee^{n-k-1} b^X_{n-k-1}(\ee) \lambda + \ee^{n-k} b^X_{n-k}(\ee).
\nonumber
\end{align}
This can be factored as
\begin{align}
\det(\lambda I - C(\ee)) &=
\left( \lambda^k + \left[ a^X_1(0) + \cO(\ee) \right] \lambda^{k-1} + \cdots
+ \left[ a^X_k(0) + \cO(\ee) \right] \right) \nonumber \\
&\quad \times \bigg( \lambda^{n-k}
+ \ee \left[ \frac{b^X_1(0)}{a^X_k(0)} + \cO(\ee) \right] \lambda^{n-k-1}
+ \cdots \nonumber \\
&\quad+ \ee^{n-k-1} \left[ \frac{b^X_{n-k-1}(0)}{a^X_k(0)} + \cO(\ee) \right] \lambda
+ \ee^{n-k} \left[ \frac{b^X_{n-k}(0)}{a^X_k(0)} + \cO(\ee) \right] \bigg),
\nonumber
\end{align}
which is the right hand side of \eqref{eq:eigsABC} because
the coefficients of the characteristic polynomials of $A_X$ and $B_X$
are the components of the vectors $a^X(0)$ and $\frac{b^X(0)}{a^X_k(0)}$.
\end{proof}

\section{The dynamics of the full system}
\label{sec:mainthm}
\setcounter{equation}{0}

In this section we consider the SFOCF \eqref{eq:sfocf}, also written as \eqref{eq:sfocf2}, with $\ee > 0$.
Here we suppose that the right hand side of \eqref{eq:sfocf} is a $C^1$ function of $\ee$ in
some interval $[0,\ee_1]$, and denote the flow of \eqref{eq:sfocf} by $\varphi_t(z;\ee)$.

\subsection{A local forward invariant region}
\label{sub:mainthm}

The motivation for the following theorem (Theorem \ref{th:mainBound})
is that we would like to know that
attractors on $\cM_0$ do not change catastrophically as the value of $\ee$ is increased from $0$.
For this reason, given $\delta > 0$ and a compact set $\Omega \subset \mathbb{R}^{n-k}$, we consider regions
\begin{equation}
\Omega_\delta = \left\{ \left[ \begin{array}{c} x \\ \hline y \end{array} \right] \,\middle|\,
\| x - H(y_1) \| \le \delta,\, y \in \Omega \right\}.
\label{eq:Omegadelta}
\end{equation}
In particular $\Omega_0 \subset \mathbb{R}^n$ is a compact subset of $\cM_0$.
We assume, not only that $\Omega$ is a trapping region for the reduced system,
but that the vector field points {\em strictly} inwards throughout the boundary of $\Omega$.
This ensures that forward orbits do not diverge for small values of $\ee$.
Specifically we use the following definition.

\begin{definition}
Let $\Omega \subset \mathbb{R}^{n-k}$ be compact with smooth boundary $\partial \Omega$\removableFootnote{
I don't have to say that $\partial \Omega$ is orientable because $\Omega$ is compact.
The ``inside of $\partial \Omega$'' is simply the interior of $\Omega$.
}.
The set $\Omega$ is said to be a {\em strong trapping region} for \eqref{eq:limitingSlowy}
if it is a trapping region and there are no points on $\partial \Omega$
at which $\frac{d y}{d \tau}$ is tangent to $\partial \Omega$.
\end{definition}

\begin{theorem}
Suppose $\cM_0$ is globally stable
and let $\Omega \subset \mathbb{R}^{n-k}$ be a strong trapping region for \eqref{eq:limitingSlowy}.
Then there exist $M > 0$, $N \ge M$, and $\ee_{\rm max} \in (0,\ee_1]$, such that
\begin{equation}
\varphi_t(z;\ee) \in \Omega_{\ee N} \,, \quad
{\rm for~all~} \ee \in (0,\ee_{\rm max}),
{\rm ~all~} z \in \Omega_{\ee M},
{\rm ~and~all~} t \ge 0.
\label{eq:mainBound}
\end{equation}
\label{th:mainBound}
\end{theorem}

Our proof of Theorem \ref{th:mainBound}, given below,
uses the following result which is proved in Appendix \ref{app:linearGrowth}.

\begin{lemma}
For any $T > 0$ there exists $K \in \mathbb{R}$ such that
\begin{equation}
\left\| \varphi_t(z;\ee) - \varphi_t(z;0) \right\| \le K \ee t, \quad
{\rm for~all~} z \in \Omega_1,
{\rm ~all~} \ee \in [0,\ee_1],
{\rm ~and~all~} t \in [0,T].
\label{eq:linearGrowth}
\end{equation}
\label{le:linearGrowth}
\end{lemma}

\begin{proof}[Proof of Theorem \ref{th:mainBound}]
Write $\varphi_t(z;\ee) = \left[ \begin{array}{c} \phi_t(z;\ee) \\ \hline \psi_t(z;\ee) \end{array} \right]$.

Since $\partial \Omega$ is compact
and the right hand side of \eqref{eq:sfocf} varies continuously with respect to $x$ and $\ee$,
there exist $\ee_2, \delta > 0$ (with $\ee \le \ee_1$ and $\delta \le 1$)such that,
for all $\ee \in (0,\ee_2)$, all $y \in \partial \Omega$,
and all $x \in \mathbb{R}^k$ with $\left\| x - H(y_1) \right\| \le \delta$,
the vector $\frac{d y}{d \tau}$ is not tangent to $\partial \Omega$ (and points inwards).

By Lemma \ref{le:expStab},
there exist $\alpha \ge 1$ and $\beta > 0$ such that
\begin{equation}
\left\| \phi_t(z;0) - H(y_1) \right\|
\le \alpha {\rm e}^{-\beta t} \left\| x - H(y_1) \right\|, \quad
{\rm for~all~} x \in \mathbb{R}^k,
{\rm ~all~} y_1 \in \mathbb{R},
{\rm ~and~all~} t \ge 0.
\label{eq:expStabilityPWL}
\end{equation}

Let $T = \frac{1}{\beta} \ln(2 \alpha)$.
Let $K \in \mathbb{R}$ be the constant in Lemma \ref{le:linearGrowth}.
Let $M = 2 K T$ and $N = 2 \alpha M$.
Let $\ee_{\rm max} = \min \left[ \ee_2, \frac{\delta}{N} \right]$.

Choose any $\ee \in (0,\ee_{\rm max})$ and $z \in \Omega_{\ee M}$.
Then
\begin{align}
\left\| \phi_t(z;\ee) - H(y_1) \right\|
&\le \left\| \phi_t(z;\ee) - \phi_t(z;0) \right\|
+ \left\| \phi_t(z;0) - H(y_1) \right\| \nonumber \\
&\le \left\| \varphi_t(z;\ee) - \varphi_t(z;0) \right\|
+ \alpha {\rm e}^{-\beta t} \left\| x - H(y_1) \right\|.
\label{eq:mainBoundProof3}
\end{align}
Notice $\ee M \le 1$, thus $z \in \Omega_1$ and so by Lemma \ref{le:linearGrowth}
\begin{equation}
\left\| \phi_t(z;\ee) - H(y_1) \right\|
\le K \ee T + \alpha \ee M
= \left( \frac{1}{2} + \alpha \right) \ee M
\le \ee N,
\nonumber 
\end{equation}
for all $t \in [0,T]$.
Also $\ee N \le \delta$ thus $\psi_t(z;\ee)$ does not escape $\Omega$ for $t \in [0,T]$.
Thus $\varphi_t(z;\ee) \in \Omega_{\ee N}$ for all $t \in [0,T]$.
Also by \eqref{eq:mainBoundProof3},
\begin{equation}
\left\| \phi_T(z;\ee) - H(y_1) \right\|
\le K \ee T + \frac{1}{2} \ee M
= \ee M.
\nonumber
\end{equation}
Thus $\varphi_T(z;\ee) \in \Omega_{\ee M}$.
Hence $\varphi_t(z;\ee) \in \Omega_{\ee N}$ for all $t \ge 0$.
\end{proof}

\subsection{Consequences of a lack of global stability}
\label{sub:CaFr06}

\begin{figure}[b!]
\begin{center}
\setlength{\unitlength}{1cm}
\begin{picture}(6.4,4.8)
\put(0,0){\includegraphics[height=4.8cm]{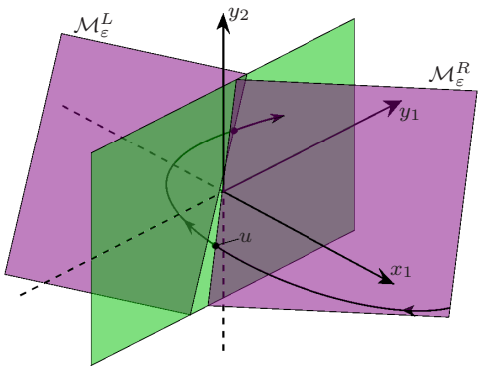}}
\put(5.17,1.3){\scriptsize $x_1$}
\put(5.29,3.4){\scriptsize $y_1$}
\put(3.02,4.72){\scriptsize $y_2$}
\put(.96,4.53){\scriptsize $\cM^L_\ee$}
\put(5.68,3.88){\scriptsize $\cM^R_\ee$}
\put(3.16,1.73){\scriptsize $u$}
\end{picture}
\caption{
A sketch of the linear, locally invariant, slow manifolds $\cM^L_\ee$ and $\cM^R_\ee$ of
the SFOCF \eqref{eq:sfocf} with $(n,k) = (3,1)$.
We also show a typical orbit; the dots indicate its intersections with
the switching manifold $x_1 = 0$.
\label{fig:linearSlowManifolds}
} 
\end{center}
\end{figure}

On each side of the switching manifold, the SFOCF
has linear, locally invariant, slow manifolds $\cM^L_\ee$ and $\cM^R_\ee$
aligned with the slow eigenspaces of $C_L(\ee)$ and $C_R(\ee)$, Fig.~\ref{fig:linearSlowManifolds}.
These slow manifolds converge to $\cM_0$ as $\ee \to 0$, in accordance with Fenichel's theorem,
but for $\ee > 0$ intersect the switching manifold on different surfaces.

If all eigenvalues of $A_L$ and $A_R$ have negative real part,
so that $\cM^L_\ee$ and $\cM^R_\ee$ are attracting,
then as a typical orbit crosses the switching manifold from right to left at a point $u$, say,
we can assume that $u$ is very near $\cM^R_\ee$.
Since $\cM_0$ is continuous at the switching manifold,
we can further assume that $u$ is an $\cO(\ee)$ distance from $\cM^L_\ee$.

If $\cM_0$ is globally stable and $\ee$ is sufficiently small,
then from $u$ the orbit rapidly approaches $\cM^L_\ee$,
as in Fig.~\ref{fig:linearSlowManifolds}.
However, if $\cM_0$ is not globally stable,
then Theorem \ref{th:mainBound} does not apply and the orbit may be repelled from $\cM^L_\ee$.
To understand this further, consider again the layer equation \eqref{eq:limitingFastx}.
As the orbit passes through $u$,
we can interpret the value of $y_1$ in \eqref{eq:limitingFastx}
as a slowly varying parameter that passes through zero.
If $\cM_0$ is not globally stable, then the size of the basin of attraction
of the equilibrium $H(y_1)$ of \eqref{eq:limitingFastx} is proportional to $|y_1|$.
Thus while the value of $y_1$ is sufficiently small,
the $x$-component of the orbit of the full system lies outside the basin of attraction of $H(y_1)$.
This instability can cause the orbit to be repelled from $\cM^L_\ee$ if $|y_1|$ does not increase too quickly.

Here we study a minimal example of this phenomenon constructed 
by choosing eigenvalues for $C_L(\ee)$ and $C_R(\ee)$
such that the matrices $A_L$, $A_R$, $B_L$, and $B_R$, have various desired properties.
First we wish all eigenvalues of $A_L$ and $A_R$ to have negative real part,
yet $H(0)$ to be an unstable equilibrium of \eqref{eq:limitingFastx} with $y_1 = 0$.
As discussed in \S\ref{sub:expstab}, this requires $k \ge 3$. 
As in \cite{CaFr05} we let the eigenvalues of $A_L$ and $A_R$ be\removableFootnote{
This is nicer than the example of \cite{CaFr06}.
}
\begin{equation}
\begin{aligned}
\lambda^L_1 &= -0.6, &
\lambda^R_1 &= -3, \\
\lambda^L_{2,3} &= -0.2 \pm {\rm i}, & 
\lambda^R_{2,3} &= -0.1 \pm 5 {\rm i}.
\end{aligned}
\label{eq:eigCanardFast}
\end{equation}
Second we wish the reduced system on $\cM_0$ to have an attractor that involves both sides of the switching manifold.
This requires $n - k \ge 2$ (otherwise the attractor can only be an equilibrium).
So that the reduced system with $\mu > 0$ has a stable limit cycle,
we let the eigenvalues of $B_L$ and $B_R$ be
\begin{align}
\nu^L_{1,2} &= -3 \pm {\rm i}, &
\nu^R_{1,2} &= 1 \pm 2 {\rm i}.
\label{eq:eigCanardSlow}
\end{align}
We then define the entries of the first columns of $C_L(\ee)$ and $C_R(\ee)$
so that the eigenvalues of these matrices are
\eqref{eq:eigCanardFast} and $\ee$ times \eqref{eq:eigCanardSlow}\removableFootnote{
A simpler alternative is to define $a_X(\ee)$ and $b_X(\ee)$ as constants,
but in this case I couldn't get such nice numerics.
Specifically, to avoid divergence in this case I had to modify the example of \cite{CaFr05}
and the small oscillations were still too small to see ({\sc ppSlowFast3.m}).
}.

\begin{figure}[b!]
\begin{center}
\setlength{\unitlength}{1cm}
\begin{picture}(12.5,4)
\put(0,0){\includegraphics[height=4cm]{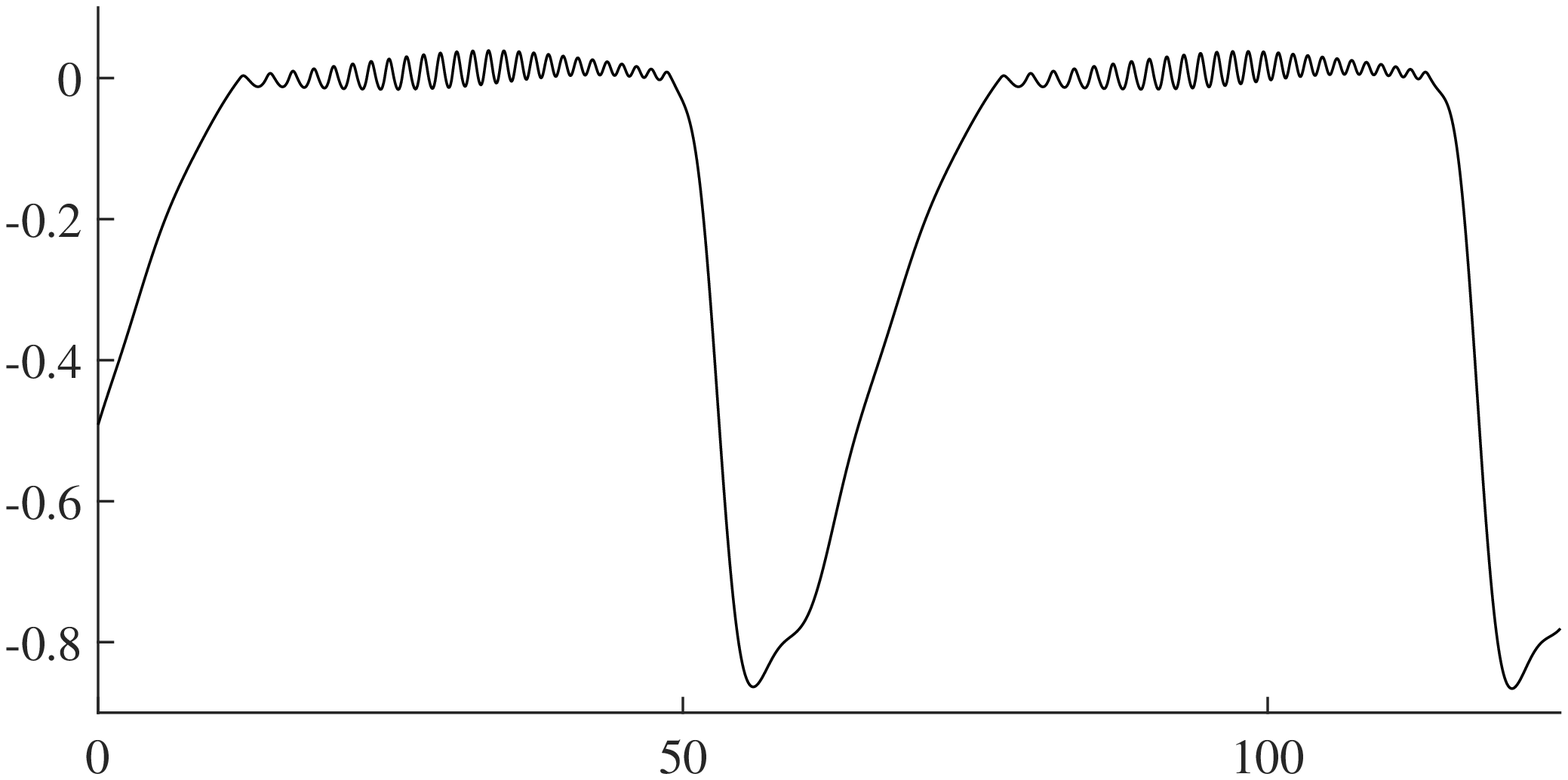}}
\put(8.5,0){\includegraphics[height=4cm]{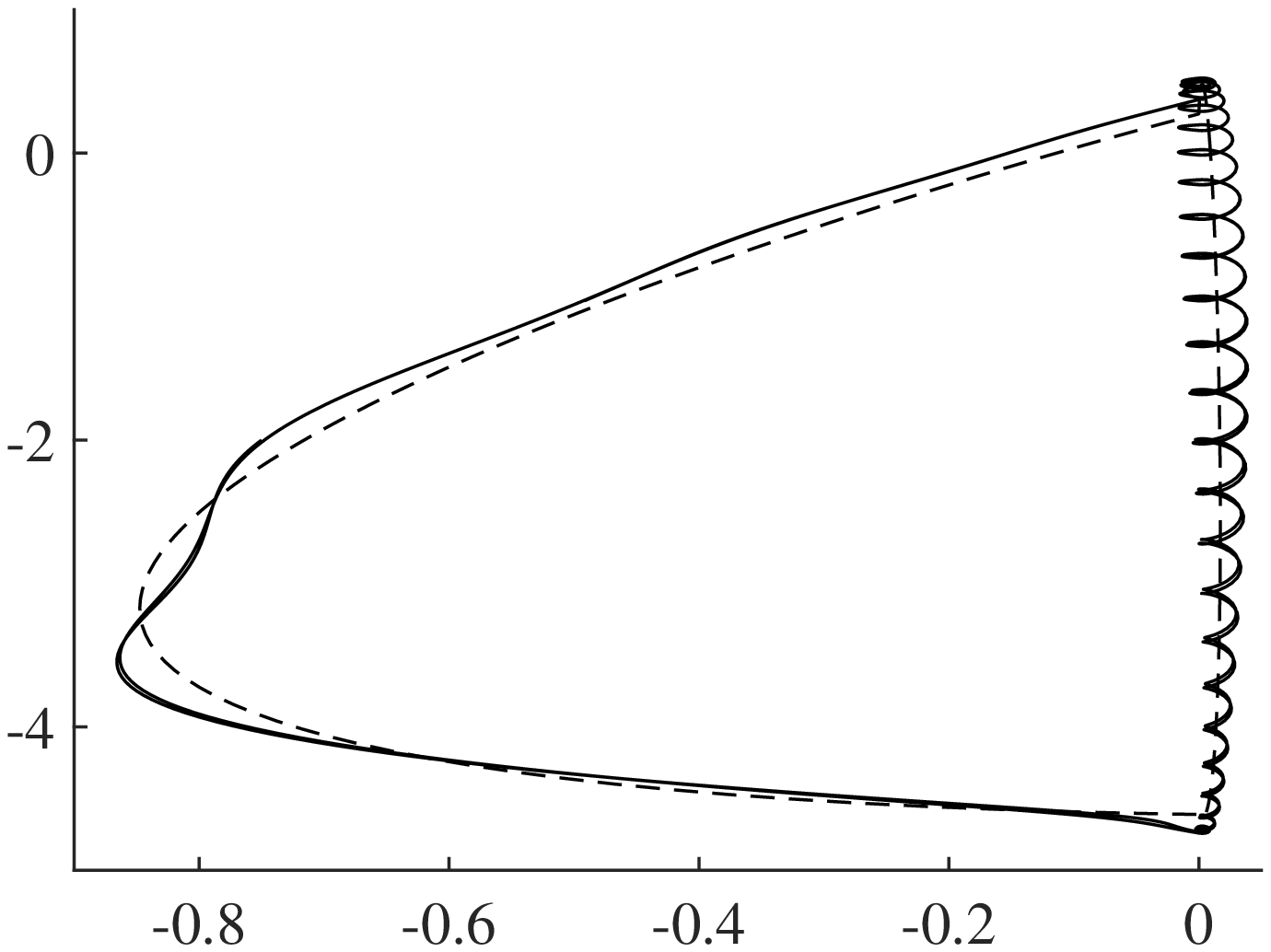}}
\put(.2,3.9){\sf \bfseries A}
\put(8.5,3.9){\sf \bfseries B}
\put(4.2,0){\small $t$}
\put(0,2.5){\small $x_1$}
\put(11.35,0){\small $x_1$}
\put(8.5,2.78){\small $y_2$}
\end{picture}
\caption{
A time series and phase portrait of the SFOCF \eqref{eq:sfocf} in five dimensions with $\ee = 0.05$ and $\mu = 1$.
For $X = L,R$, the eigenvalues of $C_X(\ee)$ are $\lambda^X_{1,2,3}$ \eqref{eq:eigCanardFast}
and $\ee \nu^X_{1,2}$ \eqref{eq:eigCanardSlow}.
Thus $(n,k) = (5,3)$ (three fast variables, two slow variables).
The dashed curve in panel B is the stable limit cycle of the reduced system \eqref{eq:limitingSlowy} on $\cM_0$.
\label{fig:canard5d}
} 
\end{center}
\end{figure}

Numerical simulations of this system suggest that it has no bounded attractor
for sufficiently small values of $\ee > 0$.
With $\ee = 0.05$, however, typical orbits remain near $\cM_0$, see Fig.~\ref{fig:canard5d}.
This is due to competition between the attracting limit cycle of the reduced system
and the repelling attractor at infinity of the layer equations.
The value of $\ee = 0.05$ is sufficiently large that the attraction dominates the repulsion and a bounded attractor exists.

A further analysis of this system is beyond the scope of this paper.
The layer equations have an unstable invariant set bounding the basin of attraction of $H(y_1)$
that, in the full system, manifests as a repelling slow manifold.
As the value of $\ee$ is decreased from $0.05$,
we expect the attractor to be destroyed through the creation of canards.

\section{Dimension reduction of BEBs}
\label{sec:app}
\setcounter{equation}{0}

By Theorem \ref{th:mainBound}, the SFOCF \eqref{eq:sfocf}
is forward invariant in the neighbourhood $\Omega_{\ee N}$ of $\cM_0$.
In this neighbourhood, $x = H(y_1) + \cO(\ee)$.
By substituting this into \eqref{eq:sfocf}, we find that on the slow time-scale $\tau = \ee t$ we have
\begin{equation}
\frac{d y}{d \tau} = \begin{cases}
\left( B_L + \cO(\ee) \right) y + e_{n-k} \mu, & y_1 \le 0, \\
\left( B_R + \cO(\ee) \right) y + e_{n-k} \mu, & y_1 \ge 0.
\end{cases}
\label{eq:slowy}
\end{equation}
That is, the dynamics in $\Omega_{\ee N}$ is governed by a regular $\cO(\ee)$ perturbation
of \eqref{eq:limitingSlowy}.

Now consider a BEB in an $n$-dimensional system for which is it not necessarily clear
which variables are fast, or even how many are fast.
We can evaluate the Jacobian matrices of the two relevant components of the system at the bifurcation,
and write their eigenvalues 
as $\lambda^X_1,\ldots,\lambda^X_k,\ee \nu^X_i,\ldots,\ee \nu^X_{n-k}$,
for values of $k$ and $\ee$ that seem sensible.
We can then use the values $\nu^X_i,\ldots,\nu^X_{n-k}$
to construct the $(n-k)$-dimensional observer canonical form.
The idea is that this reduced system can provide
a good qualitative approximation to the dynamics of the full system near the bifurcation.
Moreover, the dynamics could be compared quantitatively if the coordinate transformations are derived explicitly.
In the next section we demonstrate this dimension reduction methodology with an example.

\subsection{Ocean circulation}
\label{sub:ocean}

Here we study the ocean circulation model of \cite{RoSa17}
\begin{equation}
\begin{split}
\dot{\overline{x}} &= (1-\overline{x}) - \ee A \overline{x} |\overline{x}-\overline{y}|, \\
\dot{\overline{y}} &= \ee \left( \overline{\mu} - \overline{y} - A \overline{y} |\overline{x}-\overline{y}| \right), \\
\dot{\overline{\mu}} &= \ee \delta \left( \lambda_0 + a \overline{x} - b \overline{y} \right),
\end{split}
\label{eq:RoSa17}
\end{equation}
where bars have been added to avoid confusion with the notation that has already been developed.
The variables $\overline{x}$ and $\overline{y}$ represent
the difference in temperature and salinity of the ocean near the equator compared to near the poles,
and $\overline{\mu}$ is a forcing ratio.
The system is piecewise-smooth due to the assumption that the motion
depends only on the magnitude of the circulation, not its direction.
The small parameter $\ee$ represents the ratio of the relaxation rate for salinity
to the relaxation rate for temperature.
The parameter $\delta$ is also small, thus \eqref{eq:RoSa17} has potentially three distinct time-scales,
but here we only consider the time-scale separation effect of $\ee$.
The remaining quantities $a$, $b$, $\lambda_0$, and $A$ are scalar parameters.

In \cite{RoSa17} the authors perform a detailed study of the nonlinear dynamics
of the reduced system defined by the limit $\ee \to 0$.
In particular they show that a stable limit cycle is created via two types of BEB.
A small amplitude oscillation is created in a Hopf-like bifurcation \cite{SiMe07},
and a relaxation oscillation is created in a bifurcation governed
by both local and global properties of the system \cite{DeFr13,RoGl14}.
Here we explain these features in the full system \eqref{eq:RoSa17} via a dimension reduction analysis of the BEBs.


\begin{figure}[b!]
\begin{center}
\setlength{\unitlength}{1cm}
\begin{picture}(12.5,4)
\put(0,0){\includegraphics[height=4cm]{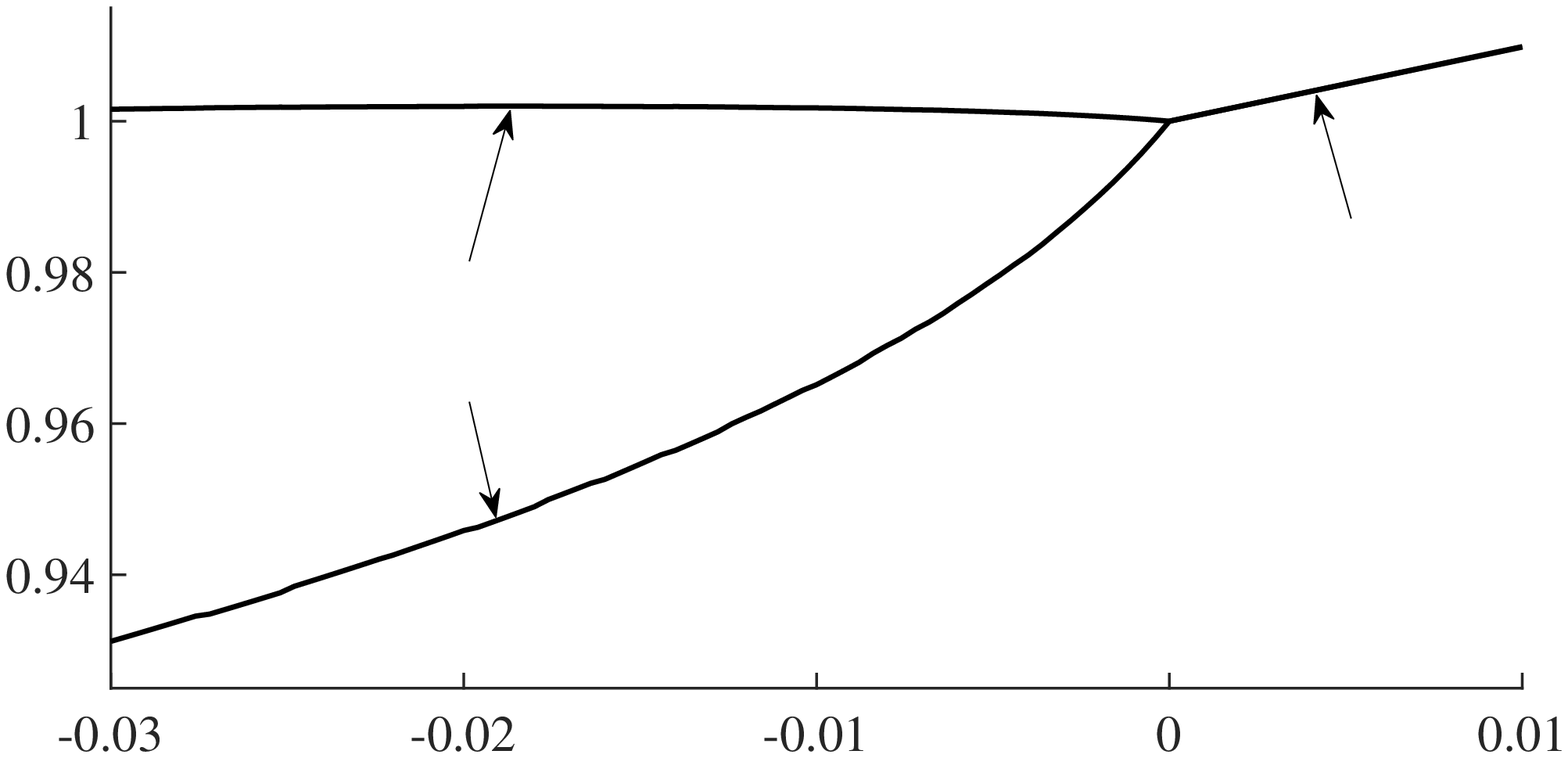}}
\put(8.5,0){\includegraphics[height=4cm]{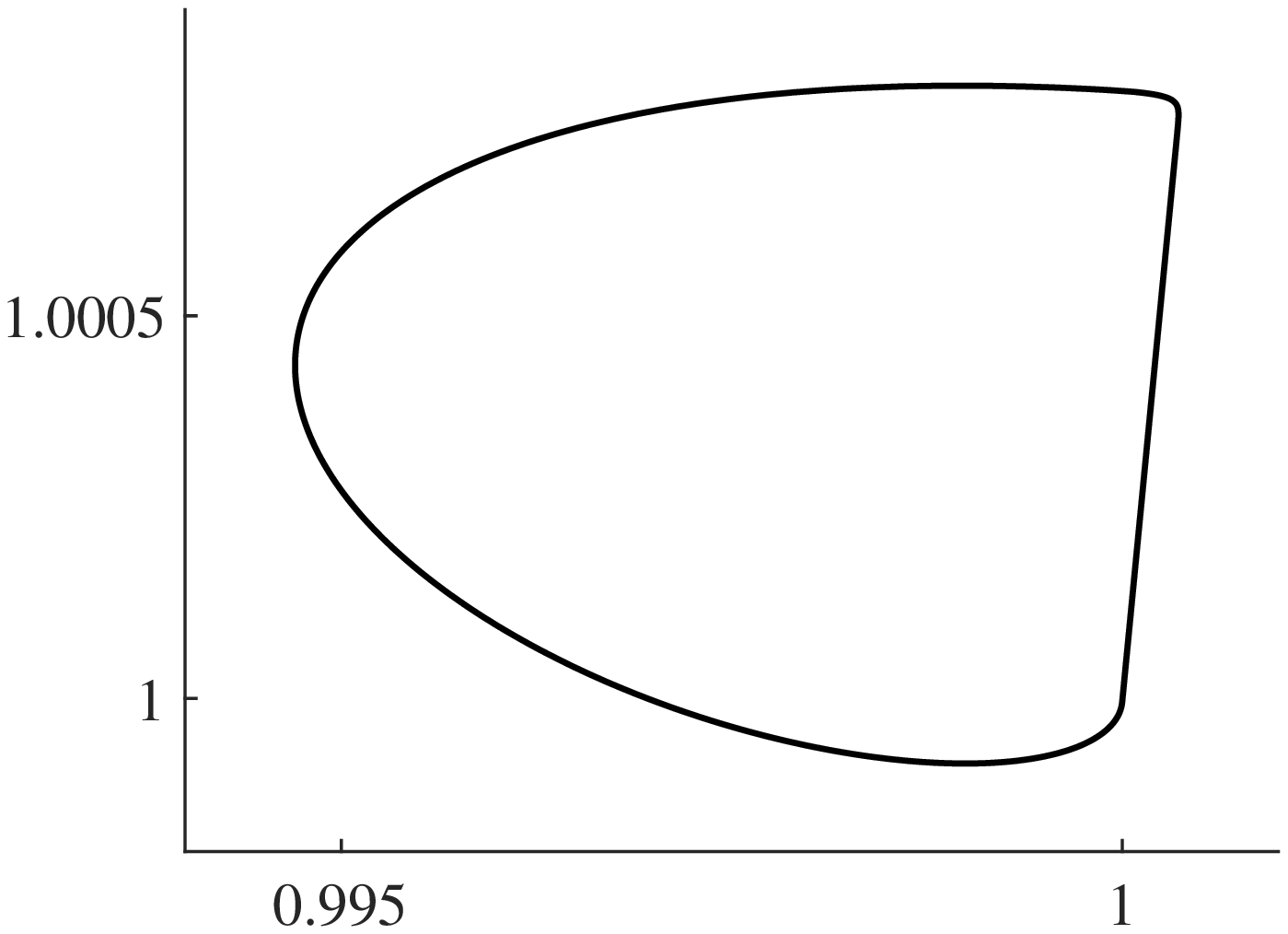}}
\put(.2,3.9){\sf \bfseries A}
\put(8.5,3.9){\sf \bfseries B}
\put(5,0){\small $\lambda_0$}
\put(0,2.3){\small $\overline{y}$}
\put(10.6,4){\footnotesize $\lambda_0 = -0.001$}
\put(11.6,0){\small $\overline{y}$}
\put(9,3.4){\small $\overline{\mu}$}
\put(2,2.54){\scriptsize limit cycle}
\put(2.19,2.23){\scriptsize bounds}
\put(6.3,2.7){\scriptsize equilibrium}
\end{picture}
\caption{
Panel A: A bifurcation diagram of \eqref{eq:RoSa17} with
\eqref{eq:RoSa17Param1}, $\ee = 0.01$, and $A = 1.1$.
Panel B: The stable limit cycle with $\lambda_0 = -0.001$.
\label{fig:bifDiag_lambda_0}
} 
\end{center}
\end{figure}

We fix\removableFootnote{
In \cite{RoSa17} the values of $a$ and $b$ are never specified.
}
\begin{align}
a &= 1, &
b &= 1, &
\delta &= 0.01,
\label{eq:RoSa17Param1}
\end{align}
with which \eqref{eq:RoSa17} has a BEB at $(\overline{x},\overline{y},\overline{\mu}) = (1,1,1)$ when $\lambda_0 = 0$.
Fig.~\ref{fig:bifDiag_lambda_0}-A shows a bifurcation diagram illustrating the BEB for $\ee = 0.01$ and $A = 1.1$.
A stable equilibrium bifurcates to a stable limit cycle as the value of $\lambda_0$ is decreased.
This limit cycle is shown in Fig.~\ref{fig:bifDiag_lambda_0}-B for $\lambda = -0.001$.

\begin{figure}[b!]
\begin{center}
\setlength{\unitlength}{1cm}
\begin{picture}(8,4)
\put(0,0){\includegraphics[width=8cm]{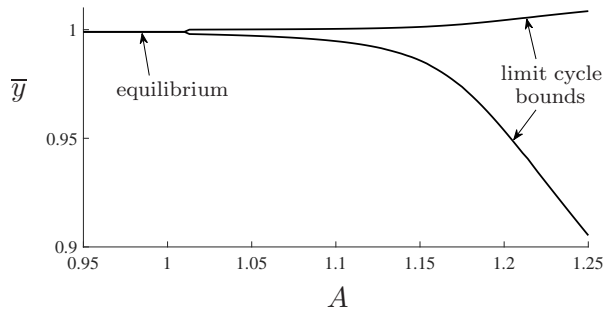}}
\put(4.2,0){\small $A$}
\put(0,2.8){\small $\overline{y}$}
\put(6.51,3.02){\scriptsize limit cycle}
\put(6.7,2.71){\scriptsize bounds}
\put(1.4,2.81){\scriptsize equilibrium}
\end{picture}
\caption{
A bifurcation diagram of \eqref{eq:RoSa17} with
\eqref{eq:RoSa17Param1}, $\ee = 0.01$, and $\lambda_0 = -0.001$.
\label{fig:bifDiag_A}
} 
\end{center}
\end{figure}

Now if we fix $\lambda = -0.001$ and vary the value of $A$,
the size of the limit cycle changes in a nonlinear fashion.
This is shown in Fig.~\ref{fig:bifDiag_A}.
The limit cycle exists for $A > 1.01$, approximately.
Over an intermediate range of values of $A$ (say, $1.01 < A < 1.15$) the limit cycle is small.
For larger values of $A$ the size of the limit cycle increases rapidly.

To explain these observations we first employ our dimension reduction methodology numerically.
With $\ee = 0.01$ and $A = 1.1$, the two sets of eigenvalues at the BEB
are $\lambda^X_1$, $\ee \nu^X_2$, and $\ee \nu^X_3$, for $X = L,R$, where\removableFootnote{
Computed with {\sc ppEig.m}.
}
\begin{equation}
\begin{aligned}
\lambda^L_1 &= -0.9888, & \lambda^R_1 &= -1.011, \\
\nu^L_1 &= -2.108, & \nu^R_1 &= 0.04402 + 0.08919 {\rm i}, \\
\nu^L_2 &= -0.004798, & \nu^R_2 &= 0.04402 + 0.08919 {\rm i},
\end{aligned}
\label{eq:oceanEigs}
\end{equation}
to four significant figures.
Fig.~\ref{fig:bifDiag_ocean_PWL}-A shows a bifurcation diagram of the two-dimensional
reduced system \eqref{eq:limitingSlowy} where we have used the eigenvalues $\nu^L_{1,2}$ and $\nu^R_{1,2}$
to construct the companion matrices $B_L$ and $B_R$.
As expected the dynamics of the reduced system (Fig.~\ref{fig:bifDiag_ocean_PWL}) is qualitatively similar
to that of full system (Fig.~\ref{fig:bifDiag_lambda_0}).

\begin{figure}[b!]
\begin{center}
\setlength{\unitlength}{1cm}
\begin{picture}(12.5,4)
\put(0,0){\includegraphics[height=4cm]{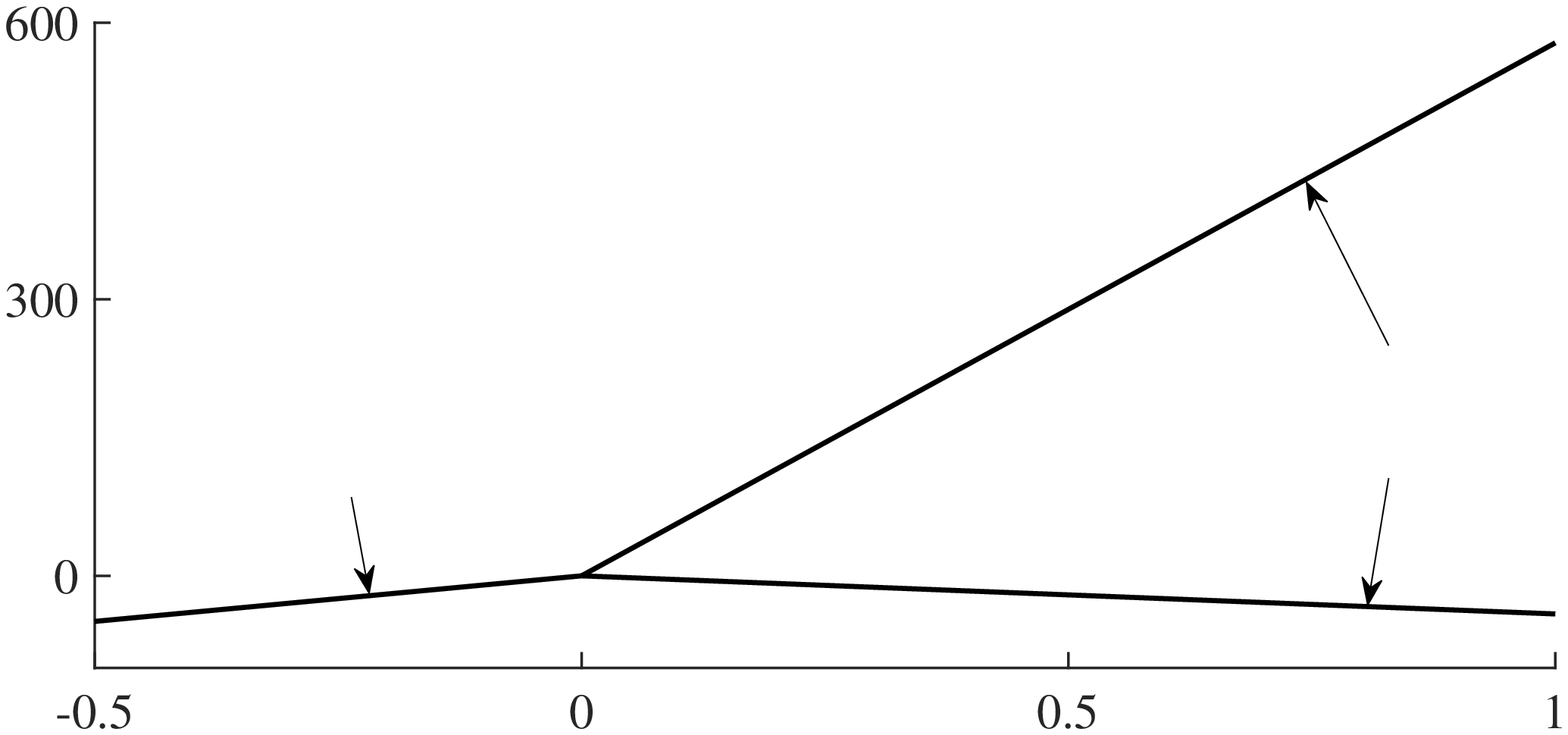}}
\put(8.5,0){\includegraphics[height=4cm]{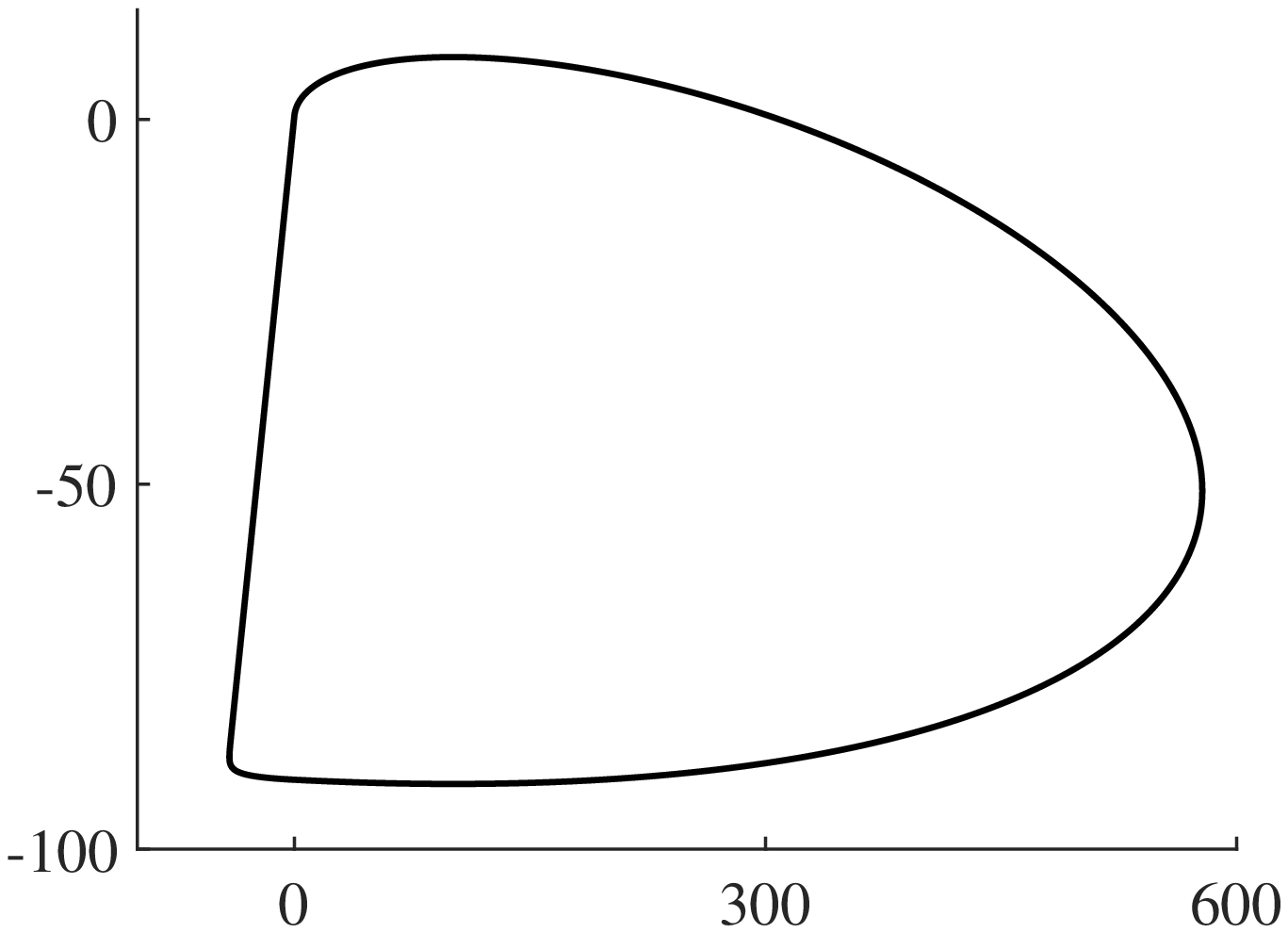}}
\put(0,3.9){\sf \bfseries A}
\put(8.5,3.9){\sf \bfseries B}
\put(4.2,0){\small $\mu$}
\put(0,3.1){\small $y_1$}
\put(11.3,4){\footnotesize $\mu = 1$}
\put(11.55,0){\small $y_1$}
\put(8.5,2.8){\small $y_2$}
\put(6.51,2.03){\scriptsize limit cycle}
\put(6.7,1.72){\scriptsize bounds}
\put(1.4,1.65){\scriptsize equilibrium}
\end{picture}
\caption{
Panel A: A bifurcation diagram of the reduced system \eqref{eq:limitingSlowy}
where $B_L$ and $B_R$ are of the form \eqref{eq:BLBR}
and have eigenvalues $\nu^L_{1,2}$ and $\nu^R_{1,2}$ given by \eqref{eq:oceanEigs}.
Panel B: The limit cycle for $\mu = 1$.
\label{fig:bifDiag_ocean_PWL}
} 
\end{center}
\end{figure}

To understand the effects of the parameter $A$, we next study the reduced system analytically.
Let $f_L(\overline{x},\overline{y},\overline{\mu})$
and $f_R(\overline{x},\overline{y},\overline{\mu})$ denote the smooth components of the right hand side of \eqref{eq:RoSa17}.
At $(\overline{x},\overline{y},\overline{\mu}) = (1,1,1)$ (where the BEB occurs),
the characteristic polynomials of $D f_L$ and $D f_R$ are
\begin{equation}
\det \left( \lambda I - D f_{L,R} \right) =
\lambda^3 + \left( 1 + \cO(\ee) \right) \lambda^2
+ \left( (1 \pm A) \ee + \cO \left( \ee^2 \right) \right) \lambda
+ b \delta \ee^2 + \cO \left( \ee^3 \right).
\end{equation}
Therefore the reduced system \eqref{eq:limitingSlowy}
that corresponds to this BEB has matrices $B_L$ and $B_R$ given by the formulas \eqref{eq:BLBR} using
\begin{equation}
\begin{aligned}
a^L(0) &= 1, &
a^R(0) &= 1, \\
b^L(0) &= \left[ \begin{array}{c} 1+A \\ b \delta \end{array} \right], &
b^R(0) &= \left[ \begin{array}{c} 1-A \\ b \delta \end{array} \right].
\end{aligned}
\nonumber
\end{equation}

As explained in \cite{FrPo98,SiMe12},
the dynamics of two-dimensional, PWL continuous systems, such as the reduced system in this example,
can be determined by classifying each component of the system
as either an attracting or repelling focus or node, or a saddle.
The $y_1 \le 0$ component of the reduced system is an attracting node,
and if $1 < A < 1 + 2 \sqrt{b \delta}$ (here $1 + 2 \sqrt{b \delta} = 1.2$)
then the $y_1 \ge 0$ component is a repelling focus.
The theory of \cite{FrPo98,SiMe12} tells us that in
this scenario a small amplitude oscillation is created in a Hopf-like bifurcation.
If $A > 1 + 2 \sqrt{b \delta}$, then the $y_1 \ge 0$ component is a repelling node.
In this case no limit cycle is created locally due to the presence of real-valued eigenvectors.
This analysis tells us that, in the limit $\ee \to 0$,
a limit cycle is created locally only for $1 < A < 1.2$.
This explains the nonlinear behaviour observed in Fig.~\ref{fig:bifDiag_A}.
For $A > 1.2$ a relaxation oscillation is created in the full system due to global features.

\section{Discussion}
\label{sec:conc}
\setcounter{equation}{0}

Although bifurcation theory for piecewise-smooth dynamical systems has matured greatly in recent years,
there remains a critical need to understand the bifurcations of such systems when the number of dimensions is large.
This paper makes a step to addressing this problem by giving conditions under which
dimension reduction is possible for BEBs of continuous systems.

The main result (Theorem \ref{th:mainBound}) is akin to Fenichel's theorem \cite{Fe71}
in that it gives conditions under which the slow dynamics
evolves according to a regular perturbation of the reduced system.
Whereas Fenichel's theorem also guarantees the existence of a slow manifold $\cM_\ee$ near the critical manifold $\cM_0$,
here we have only been able to demonstrate the existence of a forward invariant set $\Omega_{\ee N}$ near $\cM_0$
because normal hyperbolicity is not satisfied on the switching manifold.

It remains to understand what invariant objects can exist within $\Omega_{\ee N}$,
as this may have important consequences for larger values of $\ee$,
and to more completely understand the unique effects that the switching manifold can create
by proving or disproving Conjecture \ref{cj:globalStability} and
more thoroughly investigating the mixed-mode oscillations shown in \S\ref{sub:CaFr06}.
Throughout this paper we have assumed that $\cM_0$ is attracting so that the main result can be achieved.
If $\cM_0$ is repelling then the existence of a backwards invariant set can be demonstrated via a simple time-reversal,
but it remains to determine what can be said in cases for which $\cM_0$ is of saddle-type.

\appendix
\section{Proof of Proposition \ref{pr:ocf}}
\label{app:ocf}
\setcounter{equation}{0}

We evidently have $e_1^{\sf T} Q = e_1^{\sf T}$ and $e_1^{\sf T} d = 0$.
Thus $z_1 = \tilde{z}_1$, and so by directly applying the transformation
\eqref{eq:transocf} to \eqref{eq:sfpwl2} we obtain
\begin{equation}
\dot{z} = \begin{cases}
Q P_L Q^{-1} z + \frac{1}{s} \left( Q c - Q P_L Q^{-1} d \right) \mu, & z_1 \le 0, \\
Q P_R Q^{-1} z + \frac{1}{s} \left( Q c - Q P_R Q^{-1} d \right) \mu, & z_1 \ge 0.
\end{cases}
\label{eq:ocfProof1}
\end{equation}

First we show that
\begin{equation}
Q P_L Q^{-1} = J_n - p^L e_1^{\sf T}.
\label{eq:ocfProof2}
\end{equation}
Direct calculations yield
\begin{align}
J_n Q &= \left[ \begin{array}{c}
p^L_1 e_1^{\sf T} + e_1^{\sf T} P_L \\
p^L_2 e_1^{\sf T} + p^L_1 e_1^{\sf T} P_L + e_1^{\sf T} P_L^2 \\
\vdots \\
p^L_{n-1} e_1^{\sf T} + \cdots + p^L_1 e_1^{\sf T} P_L^{n-2} + e_1^{\sf T} P_L^{n-1} \\
0 \end{array} \right],
\nonumber \\
Q P_L &= \left[ \begin{array}{c}
e_1^{\sf T} P_L \\
p^L_1 e_1^{\sf T} P_L + e_1^{\sf T} P_L^2 \\
\vdots \\
p^L_{n-1} e_1^{\sf T} P_L + \cdots + p^L_1 e_1^{\sf T} P_L^{n-1} + e_1^{\sf T} P_L^n
\end{array} \right].
\nonumber
\end{align}
Thus
\begin{equation}
J_n Q - Q P_L = \left[ \begin{array}{c}
p^L_1 e_1^{\sf T} \\
\vdots \\
p^L_{n-1} e_1^{\sf T} \\
-e_1^{\sf T} \left( P_L^n + p^L_1 P_L^{n-1} + \cdots + p^L_{n-1} P_L \right)
\end{array} \right].
\label{eq:ocfProof3}
\end{equation}
By the Cayley-Hamilton theorem,
$P_L^n + p^L_1 P_L^{n-1} + \cdots + p^L_{n-1} P_L + p^L_n I$ is the zero matrix,
thus the last component of \eqref{eq:ocfProof3} is $p^L_n e_1^{\sf T}$.
Thus $J_n Q - Q P_L = p^L e_1^{\sf T}$, which verifies \eqref{eq:ocfProof2}.

Next we show that
\begin{equation}
Q P_R Q^{-1} = J_n - p^R e_1^{\sf T}.
\label{eq:ocfProof1b}
\end{equation}
Recall, since \eqref{eq:sfpwl2} is continuous, $P_L$ and $P_R$ differ in only their first columns.
Thus $P_R = P_L + \xi e_1^{\sf T}$, for some $\xi \in \mathbb{R}^n$, and so
$Q P_R Q^{-1} = Q P_L Q^{-1} + Q \xi e_1^{\sf T} Q^{-1}$.
Then substituting \eqref{eq:ocfProof2} and $e_1^{\sf T} Q^{-1} = e_1^{\sf T}$
(a consequence of $e_1^{\sf T} Q = e_1^{\sf T}$) gives
$Q P_R Q^{-1} = J_n + (Q \xi - p^L) e_1^{\sf T}$.
This shows that $Q P_R Q^{-1}$ is a companion matrix.
Characteristic polynomials are invariant under similarity transformations,
thus $Q P_R Q^{-1}$ must be the companion matrix $J_n - p^R e_1^{\sf T}$, as in \eqref{eq:ocfProof1b}.

Finally we show that
\begin{equation}
\frac{1}{s} \left( Q c - Q P_X Q^{-1} d \right) = e_n \,,
\label{eq:ocfProof4}
\end{equation}
for both $X = L$ and $X = R$.
By \eqref{eq:ocfProof2} and the definition of $d$ \eqref{eq:Qds}, we have
$Q P_X Q^{-1} d = \left( J_n - p^X e_1^{\sf T} \right) J_n^{\sf T} Q c$.
Substituting $J_n J_n^{\sf T} = I - e_n e_n^{\sf T}$ and $e_1^{\sf T} J_n^{\sf T} = 0$ gives
$Q P_X Q^{-1} d = \left( I - e_n e_n^{\sf T} \right) Q c$.
This produces \eqref{eq:ocfProof4} upon also applying the definition of $s$ \eqref{eq:Qds}.

This establishes that \eqref{eq:ocfProof1} simplifies to \eqref{eq:ocf} as required.
\hfill $\Box$

\section{Proof of Proposition \ref{pr:sfocf}}
\label{app:sfocf}
\setcounter{equation}{0}

Since $z_1 = \tilde{z}_1$, by directly applying the
transformation \eqref{eq:transsfocf} to \eqref{eq:sfpwl2} we obtain
\begin{equation}
\dot{z} = \begin{cases}
E^{-1} Q P_L Q^{-1} E z + \frac{\ee^{n-k}}{s} \left( E^{-1} Q c - E^{-1} Q P_L Q^{-1} d \right) \mu, & z_1 \le 0, \\
E^{-1} Q P_R Q^{-1} E z + \frac{\ee^{n-k}}{s} \left( E^{-1} Q c - E^{-1} Q P_R Q^{-1} d \right) \mu, & z_1 \ge 0.
\end{cases}
\label{eq:sfocfProof1}
\end{equation}
To show that \eqref{eq:sfocfProof1} simplifies to \eqref{eq:sfocf} 
we simply insert formulas established in the proof of Proposition \ref{pr:ocf}.
By \eqref{eq:ocfProof2} and \eqref{eq:ocfProof1b},
$Q P_L Q^{-1}$ and $Q P_R Q^{-1}$ are companion matrices.
Earlier we remarked that $E C_X E^{-1}$ is a companion matrix,
thus we must have $E^{-1} Q P_L Q^{-1} E = C_L$ and $E^{-1} Q P_R Q^{-1} E = C_R$.
Also, by \eqref{eq:ocfProof4},
\begin{equation}
\frac{\ee^{n-k}}{s} \left( E^{-1} Q c - E^{-1} Q P_L Q^{-1} d \right) = \ee^{n-k} E^{-1} e_n \,,
\label{eq:sfocfProof2}
\end{equation}
for both $X = L$ and $X = R$.
Since $E^{-1} e_n = \frac{1}{\ee^{n-k-1}} e_n$,
the expression \eqref{eq:sfocfProof2} reduces to $\ee e_n$,
which completes our demonstration of \eqref{eq:sfocf}.
\hfill $\Box$

\section{Proof of Lemma \ref{le:expStab}}
\label{app:expStab}
\setcounter{equation}{0}

Let $f(x;y_1)$ denote the right-hand side of \eqref{eq:limitingFastx},
and let $\overline{B}_\delta(x)$ denote the closed ball of radius $\delta$ centred at $x$.

Suppose $\cM_0$ is globally stable.
To complete the proof we show that $\cM_0$ is globally exponentially stable
(as the converse is trivial).

We first consider \eqref{eq:limitingFastx} with $y_1 = 0$.
By assumption, $H(0) = 0$ is asymptotically stable,
thus there exists $\delta > 0$ (with $\delta \le 1$) such that
\begin{equation}
\left\| \phi_t(x;0) \right\| \le 1, \quad
{\rm for~all~} x \in \overline{B}_\delta(0), {\rm ~and~all~} t \ge 0,
\label{eq:expStabProof1}
\end{equation}
and $\phi_t(x;0) \to 0$ as $t \to \infty$ for all $x \in \overline{B}_\delta(0)$.
Moreover, this convergence is uniform because $\overline{B}_\delta(0)$ is compact
(see \cite{DiNo08,Si16d} for detailed demonstrations of this in similar contexts
through use of the Arzel\`{a}-Ascoli theorem).
Thus there exists $T > 0$ such that
\begin{equation}
\left\| \phi_T(x;0) \right\| \le \frac{\delta}{2}, \quad
{\rm for~all~} x \in \overline{B}_\delta(0).
\label{eq:expStabProof2}
\end{equation}
We now show that, with $y_1 = 0$, \eqref{eq:expStab} holds using
$\alpha = \alpha_0 = \frac{2}{\delta}$ and $\beta = \beta_0 = \frac{1}{T} \ln(2)$.

Notice $f(x;0)$ is linearly homogeneous in the sense that
$\xi f(x;0) = f(\xi x;0)$ for all $x \in \mathbb{R}^n$ and all $\xi \ge 0$.
Thus the flow is similarly linearly homogeneous:
\begin{equation}
\xi \phi_t(x;0) = \phi_t(\xi x;0), \quad
{\rm for~all~} x \in \mathbb{R}^n {\rm ~and~all~} \xi \ge 0.
\label{eq:linearHomogeneity}
\end{equation}
For any $x \in \mathbb{R}^n \setminus \{ 0 \}$,
putting $\xi = \frac{\delta}{\| x \|}$ gives $\xi x \in \overline{B}_\delta(0)$,
and so $\left\| \phi_t(\xi x;0) \right\| \le 1$ for all $t \ge 0$
and $\left\| \phi_T(\xi x;0) \right\| \le \frac{\delta}{2}$.
Thus by \eqref{eq:linearHomogeneity},
\begin{equation}
\left\| \phi_t(x;0) \right\| \le \frac{\| x \|}{\delta}, \quad
{\rm for~all~} x \in \mathbb{R}^n, {\rm ~and~all~} t \ge 0,
\label{eq:expStabProof1b}
\end{equation}
and
\begin{equation}
\left\| \phi_T(x;0) \right\| \le \frac{\| x \|}{2}, \quad
{\rm for~all~} x \in \mathbb{R}^n.
\label{eq:expStabProof2b}
\end{equation}
By repeatedly applying \eqref{eq:expStabProof1b} and \eqref{eq:expStabProof2b}
we deduce that for all $t = j T + s$, where $j \in \mathbb{Z}$ is positive and $s \in [0,T]$, we have
\begin{equation}
\left\| \phi_t(x;0) \right\| \le \frac{\| x \|}{2^j \delta}.
\label{eq:expStabProof3}
\end{equation}
Then
\begin{equation}
\frac{\| x \|}{2^j \delta}
= \frac{2^{\frac{s-t}{T}} \| x \|}{\delta}
\le \frac{2^{1 - \frac{t}{T}} \| x \|}{\delta}
= \alpha_0 {\rm e}^{-\beta_0 t} \| x \|,
\nonumber
\end{equation}
giving the desired result.

Now we consider \eqref{eq:limitingFastx} with $y_1 = 1$.
Since $H(1)$ is asymptotically stable and does not belong to the switching manifold,
there exists a neighbourhood $\Omega$ of $H(1)$ that does not intersect the switching manifold and for which
\begin{equation}
\phi_t(x;1) \in \Omega, \quad
{\rm for~all~} x \in \Omega, {\rm ~and~all~} t \ge 0.
\label{eq:expStabProof11}
\end{equation}
Since $f$ is linear in $\Omega$, \eqref{eq:expStab} is satisfied
for some $a = \hat{a}$ and $b = \hat{b}$, see for instance \cite{Me07}\removableFootnote{
Specifically Lemma 2.9.
}.

To deal with initial points outside of $\Omega$,
observe that since \eqref{eq:limitingFastx} is PWL, it is also Lipschitz.
That is, there exists $K \in \mathbb{R}$ such that
\begin{equation}
\left\| f(x;1) - f(y;1) \right\| \le K \| x - y \|, \quad
{\rm for~all~} x,y \in \mathbb{R}^n.
\label{eq:LipschitzPWL}
\end{equation}
Let $T_1 = \frac{1}{\beta_0} \ln(3 \alpha_0)$.
Let $M = 3 T_1 {\rm e}^{K T_1}$, and assume $M$ is large enough that $\Omega \subset \overline{B}_M(0)$,
see Fig.~\ref{fig:schemCompactSets}.
Since $H(1) \in \Omega$ is globally asymptotically stable and $\overline{B}_M(0)$ is compact,
there exists $T_2 \in \mathbb{R}$ such that
\begin{equation}
\phi_{T_2}(x;1) \in \Omega, \quad
{\rm for~all~} x \in \overline{B}_M(0).
\label{eq:expStabProof12}
\end{equation}

\begin{figure}[b!]
\begin{center}
\setlength{\unitlength}{1cm}
\begin{picture}(6,6)
\put(0,0){\includegraphics[height=6cm]{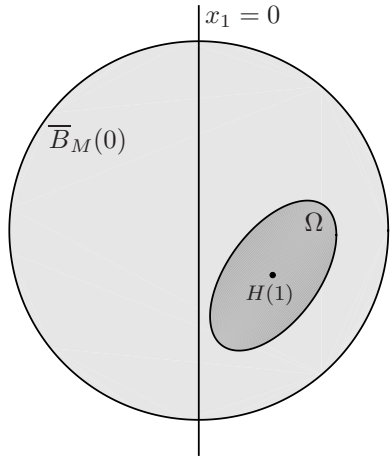}}
\put(3.08,5.76){\footnotesize $x_1 = 0$}
\put(1,4.1){\footnotesize $\overline{B}_M(0)$}
\put(4.4,3){\footnotesize $\Omega$}
\put(3.62,2.08){\scriptsize $H(1)$}
\end{picture}
\caption{
A sketch of sets introduced in the proof of Lemma \ref{le:expStab}.
\label{fig:schemCompactSets}
} 
\end{center}
\end{figure}

To deal with initial points outside $\overline{B}_M(0)$,
we approximate $f(x;1)$ with $f(x;0)$ 
and use the already established global exponential stability of $f(x;0)$.
For any $x \in \mathbb{R}^n$\removableFootnote{
This calculation could be seen as a consequence of my other Lemma.
}:
\begin{align}
\left\| \phi_{T_1}(x;1) - \phi_{T_1}(x;0) \right\|
&= \left\| \int_0^{T_1} f \left( \phi_t(x;1);1 \right)
- f \left( \phi_t(x;0);0 \right) \,dt \right\| \nonumber \\
&\le \int_0^{T_1} \left\| f \left( \phi_t(x;1);1 \right) - f \left( \phi_t(x;0);1 \right) \right\| \,dt \nonumber \\
&\quad+ \int_0^{T_1} \left\| f \left( \phi_t(x;0);1 \right) - f \left( \phi_t(x;0);0 \right) \right\| \,dt.
\label{eq:expStabProof13}
\end{align}
To the first integral in \eqref{eq:expStabProof13}
we apply the Lipschitz property \eqref{eq:LipschitzPWL}.
For the second integral observe that the integrand is simply $\| e_k \| = 1$.
Thus we have
\begin{equation}
\left\| \phi_{T_1}(x;1) - \phi_{T_1}(x;0) \right\| \le
T_1 + K \int_0^{T_1} \left\| \phi_t(x;1) - \phi_t(x;0) \right\| \,dt.
\label{eq:expStabProof14}
\end{equation}
By Gr\"{o}nwall's inequality \cite{Me07},
\begin{equation}
\left\| \phi_{T_1}(x;1) - \phi_{T_1}(x;0) \right\| \le
T_1 {\rm e}^{K T_1} = \frac{M}{3}.
\nonumber
\end{equation}
Then, by the definition of $T_1$,
\begin{equation}
\left\| \phi_{T_1}(x;1) \right\|
\le \left\| \phi_{T_1}(x;1) - \phi_{T_1}(x;0) \right\| + \left\| \phi_{T_1}(x;0) \right\|
\le \frac{M}{3} + \frac{\| x \|}{3}.
\label{eq:expStabProof15}
\end{equation}

If $\| x \| \le M$ then 
$\left\| \phi_{T_1}(x;1) \right\| \le \frac{2 M}{3}$.
Therefore, until reaching $\overline{B}_M(0)$,
the norm of $\phi_t(x;1)$ is bounded by an exponentially decaying function
(specifically $\left\| \phi_t(x;1) \right\| \le \alpha_0 {\rm e}^{\frac{-1}{T_1} \ln \left( \frac{3 \alpha_0}{2} \right) t}$).
After $\phi_t(x;1)$ enters $\overline{B}_M(0)$,
it reaches $\Omega$ within the time $T_2$, see \eqref{eq:expStabProof12},
after which is decays exponentially to $H(1)$.
This shows that, for $y_1 = 1$, \eqref{eq:expStab} holds for some
$\alpha = \alpha_1 \ge 1$ and $\beta = \beta_1 > 0$.
	
Finally we show that \eqref{eq:expStab} holds
for all $y_1 > 0$ by using the PWL nature of $f$.
Specifically, $\phi_t(x;y_1) = y_1 \phi_t \left( \frac{x}{y_1}; 1 \right)$, and so
\begin{equation}
\left\| \phi_t(x;y_1) - H(y_1) \right\|
= y_1 \left\| \phi_t \left( \frac{x}{y_1}; 1 \right) - H(1) \right\|
\le y_1 \alpha_1 {\rm e}^{\beta_1 t} \left\| \frac{x}{y_1} - H(1) \right\|
= \alpha_1 {\rm e}^{\beta_1 t} \left\| x - H(y_1) \right\|.
\nonumber
\end{equation}
By symmetry, \eqref{eq:expStab} holds for all $y_1 < 0$ for some
$\alpha = \alpha_{-1} \ge 1$ and $\beta = \beta_{-1} > 0$.
Thus \eqref{eq:expStab} holds for all $y_1 \in \mathbb{R}$
with $\alpha = \max \left[ \alpha_{-1}, \alpha_0, \alpha_1 \right]$
and $\beta = \min \left[ \beta_{-1}, \beta_0, \beta_1 \right]$,
That is, $\cM_0$ is globally exponentially stable.
\hfill $\Box$

\section{Proof of Lemma \ref{le:linearGrowth}}
\label{app:linearGrowth}
\setcounter{equation}{0}

Let
\begin{equation}
K_1 = \max_{z \in \Omega_1,\, t \in [0,T]} \left\| \varphi_t(z;0) \right\|.
\nonumber
\end{equation}
The matrices $C_L(\ee)$ and $C_R(\ee)$ are continuous functions of $\ee$ on the compact set $[0,\ee_1]$,
and so are bounded.
The matrices $C_L(\ee)$ and $C_R(\ee)$ are also differentiable at $\ee = 0$, hence the spectral norms
\begin{equation}
\frac{1}{\ee} \left\| C_L(\ee) - C_L(0) \right\|_2, \qquad
\frac{1}{\ee} \left\| C_R(\ee) - C_R(0) \right\|_2,
\nonumber
\end{equation}
are bounded on $(0,\ee_1]$ by some constant $K_2 \in \mathbb{R}$.

Let $f(z;\ee)$ denote the right hand side of \eqref{eq:sfocf}.
For any $z \in \mathbb{R}^n$ and any $\ee \in [0,\ee_1]$, 
\begin{equation}
\left\| f(z;\ee) - f(z;0) \right\|
\le \left\| \big( C_X(\ee) - C_X(0) \big) z \right\|
\le \| z \| K_2 \ee,
\label{eq:linearGrowthProof1}
\end{equation}
where, in the intermediate expression, $X = L$ if $x_1 \le 0$ and $X = R$ otherwise.
Also, $f$ is Lipschitz in $z$ and
the Lipschitz constant can be chosen independent of $\ee$
because $C_L(\ee)$ and $C_R(\ee)$ are bounded.
That is, there exists $K_3 \in \mathbb{R}$ such that
\begin{equation}
\left\| f(w;\ee) - f(z;\ee) \right\| \le K_3 \| w - z \|, \quad
{\rm for~all~} w,z \in \mathbb{R}^n,
{\rm ~and~all~} \ee \in [0,\ee_1].
\label{eq:linearGrowthProof2}
\end{equation}

Let $K = K_1 K_2 {\rm e}^{K_3 T}$.
Choose any $z \in \Omega_1$, $\ee \in [0,\ee_1]$, and $t \in [0,T]$.
Then
\begin{align}
\left\| \varphi_t(z;\ee) - \varphi_t(z;0) \right\|
&= \left\| \int_0^t f \left( \varphi_s(z;\ee);\ee \right)
- f \left( \varphi_s(z;0);0 \right) \,ds \right\| \nonumber \\
&\le \int_0^t \left\| f \left( \varphi_s(z;\ee);\ee \right)
- f \left( \varphi_s(z;0);\ee \right) \right\| \,ds \nonumber \\
&\quad+ \int_0^t \left\| f \left( \varphi_s(z;0);\ee \right)
- f \left( \varphi_s(z;0);0 \right) \right\| \,ds \nonumber \\
&= K_3 \int_0^t \left\| \varphi_s(z;\ee) - \varphi_s(z;0) \right\| \,ds
+ K_1 K_2 \ee t,
\nonumber
\end{align}
and so by Gr\"{o}nwall's inequality \cite{Me07} we have
\begin{equation}
\left\| \varphi_t(z;\ee) - \varphi_t(z;0) \right\|
\le K_1 K_2 \ee t {\rm e}^{K_3 t}
\le K \ee t,
\nonumber
\end{equation}
as required.
\hfill $\Box$


\end{document}